\documentclass[12pt]{amsart}
\usepackage{txfonts}      %{article} was 12pt latex e
\usepackage{amssymb}
\usepackage{eucal}
\usepackage{amsmath}
\usepackage{amscd}
\usepackage{xcolor}
\usepackage{multicol}
\usepackage[all]{xy}           %xypic macro for latex2.09
\usepackage{graphicx}
\usepackage{color}
\usepackage{colordvi}
\usepackage{xspace}
\usepackage{tikz}
\usepackage{makecell}
\usepackage{appendix}
\usepackage{amsthm}

\usepackage{ifpdf}
\ifpdf
\usepackage[colorlinks,final,backref=page,hyperindex]{hyperref}
\else
\usepackage[colorlinks,final,backref=page,hyperindex,hypertex]{hyperref}
\fi

\usepackage[active]{srcltx} %SRC Specials for DVI Searching

\begin{document}

%%%%%%%%%%%%%%%%%%%%%%%% Statements

\newtheorem{thm}{Theorem}[section]
\newtheorem{lem}[thm]{Lemma}
\newtheorem{cor}[thm]{Corollary}
\newtheorem{pro}[thm]{Proposition}
\theoremstyle{definition}
\newtheorem{defi}[thm]{Definition}
\newtheorem{ex}[thm]{Example}
\newtheorem{rmk}[thm]{Remark}
\newtheorem{pdef}[thm]{Proposition-Definition}
\newtheorem{condition}[thm]{Condition}

\renewcommand{\labelenumi}{{\rm(\alph{enumi})}}
\renewcommand{\theenumi}{\alph{enumi}}

\newcommand {\emptycomment}[1]{} %to remove paragraphs

\newcommand{\nc}{\newcommand}
\newcommand{\delete}[1]{}

\nc{\tred}[1]{\textcolor{red}{#1}}
\nc{\tblue}[1]{\textcolor{blue}{#1}}
\nc{\tgreen}[1]{\textcolor{green}{#1}}
\nc{\tpurple}[1]{\textcolor{purple}{#1}}
\nc{\tgray}[1]{\textcolor{gray}{#1}}
\nc{\torg}[1]{\textcolor{orange}{#1}}
\nc{\tmag}[1]{\textcolor{magenta}}
\nc{\btred}[1]{\textcolor{red}{\bf #1}}
\nc{\btblue}[1]{\textcolor{blue}{\bf #1}}
\nc{\btgreen}[1]{\textcolor{green}{\bf #1}}
\nc{\btpurple}[1]{\textcolor{purple}{\bf #1}}

\nc{\revise}[1]{\textcolor{blue}{#1}}

%%%%%%%% new symbols

\nc{\tforall}{\ \ \text{for all }}
\nc{\hatot}{\,\widehat{\otimes} \,}
\nc{\complete}{completed\xspace}
\nc{\wdhat}[1]{\widehat{#1}}

\nc{\ts}{\mathfrak{p}}
\nc{\mts}{c_{(i)}\ot d_{(j)}}

\nc{\NA}{{\bf NA}}
\nc{\LA}{{\bf Lie}}
\nc{\CLA}{{\bf CLA}}

\nc{\cybe}{CYBE\xspace}
\nc{\nybe}{NYBE\xspace}
\nc{\ccybe}{CCYBE\xspace}

\nc{\preliecom}{pre-Lie commutative\xspace}
\nc{\transpreliecom}{transposed pre-Lie commutative\xspace}
\nc{\transNovikovpoisson}{transposed Novikov-Poisson\xspace}
\nc{\transdiffnovikovpoisson}{transposed differential Novikov-poisson\xspace}
\nc{\diffnovikovpoisson}{differential Novikov-Poisson\xspace}
\nc{\preliepoisson}{pre-Lie Poisson\xspace}
\nc{\transpreliepoisson}{transposed pre-Lie poisson\xspace}

\nc{\calb}{\mathcal{B}}
\nc{\rk}{\mathrm{r}}
\newcommand{\g}{\mathfrak g}
\newcommand{\h}{\mathfrak h}
\newcommand{\pf}{\noindent{$Proof$.}\ }
\newcommand{\frkg}{\mathfrak g}
\newcommand{\frkh}{\mathfrak h}
\newcommand{\Id}{\rm{Id}}
\newcommand{\gl}{\mathfrak {gl}}
\newcommand{\ad}{\mathrm{ad}}
\newcommand{\add}{\frka\frkd}
\newcommand{\frka}{\mathfrak a}
\newcommand{\frkb}{\mathfrak b}
\newcommand{\frkc}{\mathfrak c}
\newcommand{\frkd}{\mathfrak d}
\newcommand {\comment}[1]{{\marginpar{*}\scriptsize\textbf{Comments:} #1}}

\nc{\vspa}{\vspace{-.1cm}}
\nc{\vspb}{\vspace{-.2cm}}
\nc{\vspc}{\vspace{-.3cm}}
\nc{\vspd}{\vspace{-.4cm}}
\nc{\vspe}{\vspace{-.5cm}}

%%%%%%%%%%%%%%%%%%%%%%% old symbols

\nc{\disp}[1]{\displaystyle{#1}}
\nc{\bin}[2]{ (_{\stackrel{\scs{#1}}{\scs{#2}}})}  %binomial coeff
\nc{\binc}[2]{ \left (\!\! \begin{array}{c} \scs{#1}\\
    \scs{#2} \end{array}\!\! \right )}  %binomial coeff
\nc{\bincc}[2]{  \left ( {\scs{#1} \atop
    \vspace{-.5cm}\scs{#2}} \right )}  %binomial coeff
\nc{\ot}{\otimes}
\nc{\sot}{{\scriptstyle{\ot}}}
\nc{\otm}{\overline{\ot}}
\nc{\ola}[1]{\stackrel{#1}{\la}}%${\Bbb Z}$

\nc{\scs}[1]{\scriptstyle{#1}} \nc{\mrm}[1]{{\rm #1}}

\nc{\dirlim}{\displaystyle{\lim_{\longrightarrow}}\,}
\nc{\invlim}{\displaystyle{\lim_{\longleftarrow}}\,}

\nc{\bfk}{{\bf k}} \nc{\bfone}{{\bf 1}}
\nc{\rpr}{\circ}
%\nc{\apr}{\cdot}
\nc{\dpr}{{\tiny\diamond}}
\nc{\rprpm}{{\rpr}}

%%%%%%%%%%%%%%%%%%%%% roman fonts, in alphabetic order
\nc{\mmbox}[1]{\mbox{\ #1\ }} \nc{\ann}{\mrm{ann}}
\nc{\Aut}{\mrm{Aut}} \nc{\can}{\mrm{can}}
\nc{\twoalg}{{two-sided algebra}\xspace}
\nc{\colim}{\mrm{colim}}
\nc{\Cont}{\mrm{Cont}} \nc{\rchar}{\mrm{char}}
\nc{\cok}{\mrm{coker}} \nc{\dtf}{{R-{\rm tf}}} \nc{\dtor}{{R-{\rm
tor}}}
\renewcommand{\det}{\mrm{det}}
\nc{\depth}{{\mrm d}}
\nc{\End}{\mrm{End}} \nc{\Ext}{\mrm{Ext}}
\nc{\Fil}{\mrm{Fil}} \nc{\Frob}{\mrm{Frob}} \nc{\Gal}{\mrm{Gal}}
\nc{\GL}{\mrm{GL}} \nc{\Hom}{\mrm{Hom}} \nc{\hsr}{\mrm{H}}
\nc{\hpol}{\mrm{HP}}  \nc{\id}{\mrm{id}} \nc{\im}{\mrm{im}}

\nc{\incl}{\mrm{incl}} \nc{\length}{\mrm{length}}
\nc{\LR}{\mrm{LR}} \nc{\mchar}{\rm char} \nc{\NC}{\mrm{NC}}
\nc{\mpart}{\mrm{part}} \nc{\pl}{\mrm{PL}}
\nc{\ql}{{\QQ_\ell}} \nc{\qp}{{\QQ_p}}
\nc{\rank}{\mrm{rank}} \nc{\rba}{\rm{RBA }} \nc{\rbas}{\rm{RBAs }}
\nc{\rbpl}{\mrm{RBPL}}
\nc{\rbw}{\rm{RBW }} \nc{\rbws}{\rm{RBWs }} \nc{\rcot}{\mrm{cot}}
\nc{\rest}{\rm{controlled}\xspace}
\nc{\rdef}{\mrm{def}} \nc{\rdiv}{{\rm div}} \nc{\rtf}{{\rm tf}}
\nc{\rtor}{{\rm tor}} \nc{\res}{\mrm{res}} \nc{\SL}{\mrm{SL}}
\nc{\Spec}{\mrm{Spec}} \nc{\tor}{\mrm{tor}} \nc{\Tr}{\mrm{Tr}}
\nc{\mtr}{\mrm{sk}}

%%%%%%%%%%%%%%%%%% bold face
\nc{\ab}{\mathbf{Ab}} \nc{\Alg}{\mathbf{Alg}}

%%%%%%%%%%%%%%%%%%%Bbb fonts
\nc{\BA}{{\mathbb A}} \nc{\CC}{{\mathbb C}} \nc{\DD}{{\mathbb D}}
\nc{\EE}{{\mathbb E}} \nc{\FF}{{\mathbb F}} \nc{\GG}{{\mathbb G}}
\nc{\HH}{{\mathbb H}} \nc{\LL}{{\mathbb L}} \nc{\NN}{{\mathbb N}}
\nc{\QQ}{{\mathbb Q}} \nc{\RR}{{\mathbb R}} \nc{\BS}{{\mathbb{S}}} \nc{\TT}{{\mathbb T}}
\nc{\VV}{{\mathbb V}} \nc{\ZZ}{{\mathbb Z}}

%%%%%%%%%%%%%%%%%%% cal fonts

\nc{\calao}{{\mathcal A}} \nc{\cala}{{\mathcal A}}
\nc{\calc}{{\mathcal C}} \nc{\cald}{{\mathcal D}}
\nc{\cale}{{\mathcal E}} \nc{\calf}{{\mathcal F}}
\nc{\calfr}{{{\mathcal F}^{\,r}}} \nc{\calfo}{{\mathcal F}^0}
\nc{\calfro}{{\mathcal F}^{\,r,0}} \nc{\oF}{\overline{F}}
\nc{\calg}{{\mathcal G}} \nc{\calh}{{\mathcal H}}
\nc{\cali}{{\mathcal I}} \nc{\calj}{{\mathcal J}}
\nc{\call}{{\mathcal L}} \nc{\calm}{{\mathcal M}}
\nc{\caln}{{\mathcal N}} \nc{\calo}{{\mathcal O}}
\nc{\calp}{{\mathcal P}} \nc{\calq}{{\mathcal Q}} \nc{\calr}{{\mathcal R}}
\nc{\calt}{{\mathcal T}} \nc{\caltr}{{\mathcal T}^{\,r}}
\nc{\calu}{{\mathcal U}} \nc{\calv}{{\mathcal V}}
\nc{\calw}{{\mathcal W}} \nc{\calx}{{\mathcal X}}
\nc{\CA}{\mathcal{A}}

%%%%%%%%%%%%%%%%%%  frak fonts
\nc{\fraka}{{\mathfrak a}} \nc{\frakB}{{\mathfrak B}}
\nc{\frakb}{{\mathfrak b}} \nc{\frakd}{{\mathfrak d}}
\nc{\oD}{\overline{D}}
\nc{\frakF}{{\mathfrak F}} \nc{\frakg}{{\mathfrak g}}
\nc{\frakm}{{\mathfrak m}} \nc{\frakM}{{\mathfrak M}}
\nc{\frakMo}{{\mathfrak M}^0} \nc{\frakp}{{\mathfrak p}}
\nc{\frakS}{{\mathfrak S}} \nc{\frakSo}{{\mathfrak S}^0}
\nc{\fraks}{{\mathfrak s}} \nc{\os}{\overline{\fraks}}
\nc{\frakT}{{\mathfrak T}}
\nc{\oT}{\overline{T}}
%\nc{\frakx}{{\mathfrak x}}
\nc{\frakX}{{\mathfrak X}} \nc{\frakXo}{{\mathfrak X}^0}
\nc{\frakx}{{\mathbf x}}
%\nc{\frakTxo}{{\frakTx}^0}
\nc{\frakTx}{\frakT}      %All rooted trees, correspond to \ncsha(X)
\nc{\frakTa}{\frakT^a}        % rooted trees for \ncsha(A)
\nc{\frakTxo}{\frakTx^0}   % rooted trees for \ncshao(X)
\nc{\caltao}{\calt^{a,0}}   % rooted trees for \ncshao(A)
\nc{\ox}{\overline{\frakx}} \nc{\fraky}{{\mathfrak y}}
\nc{\frakz}{{\mathfrak z}} \nc{\oX}{\overline{X}}

%%%%%%%%%%%%%%%%%%%%%%%%%%%%%%%%%%%%%%%%%%%%%%%%%%%%%%%%%%%%%%%%%%

\title[Transposed Novikov-Poisson algebras]{Transposed Novikov-Poisson algebras}

\author{Jiarou Jin}
\address{School of Mathematics, Hangzhou Normal University,
Hangzhou, 311121, China}
\email{jrjin@stu.hznu.edu.cn}

\author{Yanyong Hong (corresponding author)}
\address{School of Mathematics, Hangzhou Normal University,
Hangzhou, 311121, China}
\email{yyhong@hznu.edu.cn}



\begin{abstract}
In this paper, we introduce the definition of transposed Novikov-Poisson algebras, whose affinization are transposed Poisson algebras. Moreover, we show that there is a natural transposed Poisson algebra structure on the tensor product of a transposed Novikov-Poisson algebra and a right differential Novikov-Poisson algebra. A transposed Poisson algebra also naturally arises from a transposed Novikov-Poisson algebra by taking the commutator Lie algebra of the Novikov algebra. We show that the tensor products of two transposed Novikov-Poisson algebras are also transposed Novikov-Poisson algebras. Several constructions of transposed Novikov-Poisson algebras are presented. Moreover, transposed Novikov-Poisson algebras are closely related to $\frac{1}{2}$-derivations of the associated Novikov algebras. By using $\frac{1}{2}$-derivations, we show that there are non-trivial transposed Novikov-Poisson algebra structures on solvable Novikov algebras with some conditions. We also prove that if a non-trivial transposed Novikov-Poisson algebra is simple, then the associated Novikov algebra is simple. Therefore, if the base field is  algebraically closed and of characteristic 0, then any simple transposed Novikov-Poisson is of dimension $1$.  Transposed Novikov-Poisson algebra structures on some simple Novikov algebras are also characterized.
\end{abstract}

\maketitle

\vspace{-1.2cm}

\tableofcontents

\vspace{-1.2cm}

\allowdisplaybreaks

\section{Introduction}
\vspace{-.2cm}

%In this paper, we introduce a new algebraic structure, which is obtained via the affinization of transposed Poisson algebras.
%\subsection{Poisson algebras, transposed Poisson algebras and Novikov-Poisson algebras}\

Poisson algebras originated in the 1970s from the study of Poisson geometry \cite{poisson1,poisson2}, which now serve as essential tools in fields such as symplectic geometry, quantization theory, integrable systems, and quantum groups. \delete{The triple \( (L, \cdot, [\cdot,\cdot]) \) is called a {\bf Poisson algebra} if \( (L, \cdot) \) is a commutative associative algebra, \( (L, [\cdot,\cdot]) \) is a Lie algebra and they satisfy the following compatibility condition:
\begin{eqnarray}
&&\label{transposedpoissonlu}[x, z\cdot y] = z\cdot [x, y] + [x, z ]\cdot y \quad \text{for all}\: x, y, z \in L.
\end{eqnarray}
\quad }
A dual notion of the Poisson algebra called transposed Poisson algebra was introduced in \cite{Bai} by exchanging the roles of the two binary operations in the Leibniz rule defining the Poisson algebra. More precisely, a triple \( (L, \cdot, [\cdot,\cdot]) \) is called a {\bf transposed Poisson algebra}, if \( (L, \cdot) \) is a commutative associative algebra, \( (L, [\cdot,\cdot]) \) is a Lie algebra and the following compatibility condition holds:
\begin{eqnarray}
&&\label{transposedpoissonlu}2z \cdot [x, y] = [z \cdot x, y] + [x, z \cdot y] \quad \text{for all}\: x, y, z \in L.
\end{eqnarray}
As shown in \cite{Bai}, transposed Poisson algebras share many common properties of Poisson algebras such as closure undertaking tensor products and Koszul self-duality as an operad, and are closely related with other algebraic structures including Novikov-Poisson algebras and 3-Lie algebras.\delete{
arises naturally from a Novikov-Poisson algebra by taking the commutator Lie algebra of the Novikov algebra. Moreover, there are close relationships between transposed Poisson algebras and 3-Lie algebras (see \cite{Bai}).} They also coincide with commutative Gel'fand-Dorfman algebras \cite{S} and are instances of algebras of Jordan brackets \cite{simple-simple}. Furthermore, the classical limit of a Novikov deformation of a commutative associative algebra yields a transposed Poisson algebra \cite{CB}.
A significant connection between $\frac{1}{2}$-derivations of Lie
 algebras and transposed Poisson algebras was established in \cite{1-2}. This connection has been used to classify all transposed Poisson algebra structures on various important classes of Lie algebras, including finite-dimensional complex semisimple Lie algebras \cite{1-2}, Witt algebra \cite{1-2}, Virasoro algebra \cite{1-2},  twisted Heisenberg-Virasoro algebra \cite{YH}, Schr\"odinger-Virasoro algebra \cite{YH}, solvable and perfect Lie algebras \cite{KK}, Virasoro-type algebras \cite{KKS, ZSZ} and Lie incidence algebras \cite{KK2}. Simple transposed Poisson algebras have been studied in \cite{simple-simple} and a bialgebra theory of such algebras was given in \cite{LB}.

\delete{ A Novikov algebra was introduced from Hamiltonian operators in variational calculus~\cite{GD1, GD2} and Poisson brackets of hydrodynamic type~\cite{BN}.}

Novikov algebras first appeared in the study of Hamiltonian
operators in the formal variational calculus \cite{GD1, GD2} and Poisson brackets of hydrodynamic type \cite{BN}. By the result in \cite{X1}, they also correspond to a class of Lie conformal algebras which describe the singular part of operator product expansion of chiral fields in conformal field
theory \cite{K1}.
A {\bf (left) Novikov algebra} $(A,\circ)$ is a vector space $A$ equipped with a binary operation $\circ:A \otimes A \to A$ satisfying the following compatibility conditions:
\begin{eqnarray}
&\label{Novikov1}(x\circ y)\circ z=(x\circ z)\circ y,\\
&\label{Novikov2}(x\circ y)\circ z-x\circ (y\circ z)=(y\circ x)\circ z-y\circ (x\circ z)\;\;\;\text{for all $x$, $y$, $z\in A$.}
\end{eqnarray}
Let $B$ be a vector space with a binary operation $\diamond$. If $(B, \circ)$ is a Novikov algebra with $a\circ b:=b\diamond a$ for all $a$, $b\in B$, then $(B, \diamond)$ is called a {\bf right Novikov algebra}. By Eq. (\ref{Novikov2}), Novikov algebras are also an important subclass of pre-Lie algebras
which have close relationships with many fields in mathematics and physics such as convex homogeneous cones, affine manifolds and affine structures on Lie groups, deformation of
associative algebras, vertex algebras and so on.

The Novikov algebra also shows it significance in its role in affinization.
%The affinization of a Novikov algebra is obtained by taking its tensor product with \(k[t, t^{-1}]\), which naturally induces a Lie algebra structure.
\begin{thm}~\cite{BN}\label{thmm0}
Let {\bf k} be a field of characteristic 0 and $A$ be a vector space over {\bf k} endowed with a binary operation $\circ$. Define a binary operation on $A[t,t^{-1}]\coloneqq A\otimes \bfk[t,t^{-1}]$ by
\begin{eqnarray*}
&&[xt^m, yt^n]=m(x\circ y)t^{m+n-1}-n(y\circ x)t^{m+n-1},
\end{eqnarray*}
for all $x,y\in A$ and $m, n\in\mathbb{Z}$, where $x t^m:=x\otimes t^{m}$. Then $(A[t,t^{-1}],[\cdot,\cdot])$ is a Lie algebra if and only if $(A, \circ)$ is a Novikov algebra.
\end{thm}
\noindent This construction yields many important infinite-dimensional Lie algebras in mathematical physics, such as the Witt algebra, the centerless Heisenberg-Virasoro algebra, the centerless Schr\"odinger-Virasoro algebra and so on. The algebraic underpinning of this affinization is further clarified by the Koszul dual relationship between the Novikov and right Novikov operads \cite{operad1,operad2}.

\delete{Note that the affinization of Novikov algebras admits an operadic interpretation through the Koszul dual relationship between Novikov and right Novikov operads~\cite{operad1,operad2}. The affinization of Novikov algebras gives many of infinite-dimensional Lie algebras
which are important in mathematical physics, such as Witt algebra, centerless Heisenberg-Virasoro algebra, centerless Schr\"odinger-Virasoro algebra and so on. Motivated by the works on classifying transposed Poisson algebras on infinite-dimensional Lie algebras, it is meaningful to consider those transposed Poisson algebras which are affinizations of some algebra structures.}

Now, consider a vector space $A$ with two binary operations $\cdot,\circ$. One can define both a Lie bracket $[\cdot,\cdot]$ on $A[t,t^{-1}]$ via the Novikov product $\circ$ as above, and a commutative associative product via:
\begin{eqnarray*}
&&\label{luolang1}(xt^m)\cdot (yt^n)=(x\cdot y)t^{m+n}\;\;\text{for all $x$, $y\in A$, $m$, $n\in \mathbb{Z}$.}
\end{eqnarray*}
\delete{Obviously, $(A,\cdot)$ is a commutative associative algebra if and only if $(A[t,t^{-1}],\cdot)$ is a commutative associative algebra. We call that $(A[t,t^{-1}],\cdot)$ is an affinization of $(A,\cdot)$. Moreover,} By the result in \cite{LZ}, $(A[t,t^{-1}],\cdot, [\cdot,\cdot])$ is a Poisson algebra if and only if $(A, \cdot, \circ)$ is a differential Novikov-Poisson algebra, which was introduced in~\cite{BCZ}. Note that differential Novikov-Poisson algebras are a subclass of Novikov-Poisson algebras introduced in \cite{Xu1}.

This background motivates a natural question: what is the corresponding algebraic structure on $A$ that induces a transposed Poisson algebra under the same affinization process? To answer this, we introduce the following notion.
\delete{A natural question arises: what is the corresponding algebraic structure for transposed Poisson algebras under the affinization? This motivates us to introduce the following notion.}
\begin{defi}
The triple $(A,\cdot,\circ)$ is called a {\bf \transNovikovpoisson algebra} if $(A,\cdot)$ is a commutative associative algebra, $(A,\circ)$ is a Novikov algebra, and  the following compatibility conditions hold:
\begin{eqnarray}
&\label{transNP1}\quad(x\cdot y)\circ z=(x\cdot z)\circ y,\\
&\label{transNP2}2z\cdot(x\circ y)=(z\cdot x)\circ y+x\circ(z\cdot y)\;\;\;\text{for all $x$, $y$, $z\in A$.}
\end{eqnarray}
\end{defi}
\delete{The definition of {\transNovikovpoisson algebras} is motivated by their correspondence with transposed Poisson algebras under Novikov affinization (see Theorem~\ref{thmm2}).}
\noindent Both a Novikov algebra with the trivial commutative associative algebra and a commutative associative algebra with the trivial Novikov algebra are transposed Novikov-Poisson algebras.  We call such transposed Novikov-Poisson algebras {\bf trivial}. The essence of this definition is justified by the main result of this paper (Theorem~\ref{thmm2}), which establishes that the affinization $(A[t,t^{-1}],\cdot,[\cdot,\cdot])$ is a transposed Poisson algebra precisely when $(A,\cdot,\circ)$ is a transposed Novikov-Poisson algebra.

In fact, transposed Novikov-Poisson algebras have closer relationships with transposed Poisson algebras and have good properties. \delete{In Section 2, we first show the relationships of {\transNovikovpoisson algebras} and transposed Poisson algebras.} Besides the correspondence between transposed Novikov-Poisson algebras and transposed Poisson algebras under the affinization, we show that a transposed Poisson algebra  naturally
arises from a transposed Novikov-Poisson algebra by taking the commutator Lie algebra of the
Novikov algebra and can also be derived from a {\transNovikovpoisson algebra} via its derivations (See Theorem~\ref{thmm1}). We further show that there is a transposed Poisson algebra structure on the tensor product of a transposed Novikov-Poisson algebra and a right differential Novikov-Poisson algebra. Moreover, similar to Novikov-Poisson algebras, the tensor product of two transposed Novikov-Poisson algebras is also a transposed Novikov-Poisson algebra.  We also present several methods to construct transposed Novikov-Poisson algebras. A classification of 2-dimensional {\transNovikovpoisson algebras} over $\mathbb{C}$ is also presented and several algebraic identities of transposed Novikov-Poisson algebras are given.  We also show that there is no direct relationship between transposed Novikov-Poisson algebras and Novikov-Poisson algebras. Inspired by the work \cite{1-2}, we also present a relationship between \(\frac{1}{2}\)-derivations of the associated Novikov algebras and transposed Novikov-Poisson algebras. Moreover, a relationship between  transposed Novikov-Poisson algebras and Hom-Novikov algebras is given. By using \(\frac{1}{2}\)-derivations of Novikov algebras, we develop a technical construction of {\transNovikovpoisson algebras} on a solvable Novikov algebra $(A, \circ)$ with $\operatorname{Ann}_A(A)\neq0$ (see Corollary 3.15). \delete{The approach is inspired by the brilliant method of I. Kaygorodov, V. Lopatkin and Z. Zhang~\cite{1/2sol1}.} For the case $\operatorname{Ann}_A(A)=0$, the {\transNovikovpoisson} structure on $(A, \circ)$ may not exist. Hence we impose stronger conditions and propose an explicit  method to construct transposed Novikov-Poisson algebras on $(A, \circ)$ (see Theorem 3.20). Finally, we show that if $(A, \cdot, \circ)$ is a  non-trivial simple transposed Novikov-Poisson algebra, then $(A, \circ)$ is a simple Novikov algebra. Therefore, any simple transposed Novikov-Poisson algebra
over an algebraically closed field of characteristic 0 is of dimension 1.  Transposed Novikov-Poisson algebras on some simple Novikov algebras are also characterized.

\delete{generalize the affinization construction for transposed Poisson algebras (see Proposition~\ref{pro-constr-trp}) and investigate basic properties, identities, and examples of {\transNovikovpoisson algebras}. Among these properties, we observe that a natural {\transNovikovpoisson} structure arises on the tensor product of two such algebras?aan important feature similar to transposed Poisson algebras.
 We also present several methods for constructing new {\transNovikovpoisson algebras} from existing ones. The section includes a classification of 2-dimensional {\transNovikovpoisson algebras} over $\mathbb{C}$ and a natural construction from commutative associative algebras (see Proposition~\ref{centroids}). Finally, we note that the relationship between {\transNovikovpoisson algebras} and Novikov-Poisson algebras is independently.}

This paper is  organized as follows. In Section 2, we present the relationships between transposed Novikov-Poisson algebras and transposed Poisson algebras. Moreover, some basic properties, identities,  examples  and constructions of transposed Novikov-Poisson algebras are presented. In Section 3, a relation between \(\frac{1}{2}\)-derivations of the associated Novikov algebras and transposed Novikov-Poisson algebras is established. Applying this relation, the transposed Novikov-Poisson algebras structures on solvable Novikov algebras are investigated. Section 4 is devoted to studying simple transposed Novikov-Poisson algebras. W show that if $(A, \cdot, \circ)$ is a  non-trivial simple transposed Novikov-Poisson algebra,then $(A, \circ)$ is a simple Novikov algebra. Moreover, transposed Novikov-Poisson algebras on some simple Novikov algebras are also characterized.

\delete{In section 3, we introduce $\frac{1}{2}$-derivations of Novikov algebras. These derivations are structurally essential since left multiplication in a commutative associative algebra is a $\frac{1}{2}$-derivation of associated Novikov algebra. The relationship between \(\frac{1}{2}\)-derivations of associatived Novikov algebra and a {\transNovikovpoisson} algebra is similar to that established by B. Ferreira, I. Kaygorodov and V. Lopatkin~\cite{1-2} for Lie algebras and transposed Poisson algebras. Moreover, we develop a technical construction of {\transNovikovpoisson algebras} on a solvable Novikov algebra with $\operatorname{Ann}_A(A)\neq0$. The approach is inspired by the brilliant method of I. Kaygorodov, V. Lopatkin and Z. Zhang~\cite{1/2sol1}. For cases $\operatorname{Ann}_A(A)=0$, the {\transNovikovpoisson} structure may not exist, hence we impose stronger conditions and propose an explicit construction method.

\delete{Novikov-Poisson algebras~\cite{Xu1} were introduced by substituting the Lie bracket in a Poisson algebra with the product of a Novikov algebra. A {\bf Novikov-Poisson algebra} is a triple $(A, \cdot, \circ)$, where $(A,\cdot)$ is a commutative associative algebra, $(A,\circ)$ is a Novikov algebra, and they satisfy the following compatibility conditions:
\begin{eqnarray}
&\label{NP1} (x\cdot y) \circ z = x\cdot (y \circ z),\\
&\label{NP2}  (x \circ y)\cdot z - (y \circ x) \cdot z = x \circ (y\cdot z) - y \circ (x\cdot z),
\end{eqnarray}
for all $x$, $y$, $z\in A$. Researches on Novikov-Poisson algebras have expanded to include low-dimensional classifications~\cite{someresult}, bialgebra structures~\cite{NPbialgebra}, and relationships with quadratic and higher-dimensional Lie conformal algebras~\cite{NPconformal1,NPconformal2}.\par
\delete{In~\cite{Bai}, Bai, Bai, Guo, and Wu introduced a dual class of the Poisson algebras. The triple \( (L, \cdot, [\cdot,\cdot]) \) is called a {\bf transposed Poisson algebra} if \( (L, \cdot) \) is a commutative associative algebra, \( (L, [\cdot,\cdot]) \) is a Lie algebra and they satisfy the following compatibility condition
\begin{eqnarray}
&&\label{transposedpoissonlu}2z \cdot [x, y] = [z \cdot x, y] + [x, z \cdot y] \quad \text{for all}\: x, y, z \in L.
\end{eqnarray}
The Eq.~\eqref{transposedpoissonlu} is following by switching the roles played by the two binary operations in the Leibniz rule in a Poisson algebra and introducing a factor 2 on the left hand side. The authors show that transposed Poisson algebras are closely related to Noivkov-Poisson algebras, 3-Lie algebras and so on. The introduction of transposed Poisson algebras carries considerable significance.
\vspb}}

%\subsection{Novikov algebras and their affinization}\

%\subsection{Outline of paper}
%The paper is briefly outlined as following.\par
In Section 2, we first show the relationships of {\transNovikovpoisson algebras} and transposed Poisson algebras. Besides the correspondence under affinization, we show that a transposed Poisson algebra can also be derived from a {\transNovikovpoisson algebra} directly or via its derivations (See Theorem~\ref{thmm1}).
 We further generalize the affinization construction for transposed Poisson algebras (see Proposition~\ref{pro-constr-trp}) and investigate basic properties, identities, and examples of {\transNovikovpoisson algebras}. Among these properties, we observe that a natural {\transNovikovpoisson} structure arises on the tensor product of two such algebras?aan important feature similar to transposed Poisson algebras.
 We also present several methods for constructing new {\transNovikovpoisson algebras} from existing ones. The section includes a classification of 2-dimensional {\transNovikovpoisson algebras} over $\mathbb{C}$ and a natural construction from commutative associative algebras (see Proposition~\ref{centroids}). Finally, we note that the relationship between {\transNovikovpoisson algebras} and Novikov-Poisson algebras is independently.

In Scection 3, we introduce $\frac{1}{2}$-derivations of Novikov algebras. These derivations are structurally essential since left multiplication in a commutative associative algebra is a $\frac{1}{2}$-derivation of associated Novikov algebra. The relationship between \(\frac{1}{2}\)-derivations of associatived Novikov algebra and a {\transNovikovpoisson} algebra is similar to that established by B. Ferreira, I. Kaygorodov and V. Lopatkin~\cite{1-2} for Lie algebras and transposed Poisson algebras. Moreover, we develop a technical construction of {\transNovikovpoisson algebras} on a solvable Novikov algebra with $\operatorname{Ann}_A(A)\neq0$. The approach is inspired by the brilliant method of I. Kaygorodov, V. Lopatkin and Z. Zhang~\cite{1/2sol1}. For cases $\operatorname{Ann}_A(A)=0$, the {\transNovikovpoisson} structure may not exist, hence we impose stronger conditions and propose an explicit construction method.

In Section 4, the notion of a quasi-ideal for \transNovikovpoisson algebras is introduced. This concept is preferred because its definition aligns directly with the forms $(x\cdot y)\circ z$ or $z\circ (xy)$ found in Eqs.~\eqref{transNP1} and~\eqref{transNP2}. And we conclude that for a simple {\transNovikovpoisson algebra} has a simple associated Novikov algebra, this proof follows the ingenious idea of A. Fern\'andez Ouaridi ~\cite{simple-simple}. Furthermore, the classification of simple Novikov algebras has been extensively studied.\delete{ Osborn classified finite simple Novikov algebras with idempotent elements and their modules over an algebraically closed field of characteristic $p\geq3$~\cite{Ofinite}, Xu final give a complete classification~\cite{simplesushuyouxian}.} We consider the {\transNovikovpoisson algebras} whose Novikov algebraic structure is simple.\\}

\noindent{\bf Notations.} Throughout this paper, let  $\bf k$ be a field whose characteristic is not equal to $2$. All tensors over ${\bf k}$ are denoted by $\otimes$. Denote by $\mathbb{C}$, $\mathbb{Z}$, $\mathbb{N}$ and $\mathbb{Z}_+$ the sets of complex numbers, integer numbers, non-negative integers and positive integers respectively.
Let $A$ be a vector space with a binary operation $\ast$ and $x\in A$. Define linear maps
$L_\ast(x)$, $R_\ast(x)\in \text{End}_{\bf k}(A)$ as follows:
\begin{eqnarray*}
L_\ast(x)(y)=x\ast y,\;\;R_\ast(x)(y)=y\ast x\;\;\text{for all $y\in A$.}
\end{eqnarray*}

\section{Properties of \transNovikovpoisson algebras and the relationships with transposed Poisson algebras}
In this section, we present some examples, constructions, identities, basic properties of transposed Novikov-Poisson algebras. The relationships between transposed Novikov-Poisson algebras and transposed Poisson algebras are also given.
\subsection{Transposed Novikov-Poisson algebras and transposed Poisson algebras}\

\delete{Recall that a {\bf (left) Novikov algebra} $(A,\circ)$ is a vector space $A$ equipped with a binary operation $\circ:A \otimes A \to A$ satisfying the following compatibility conditions:
\begin{eqnarray}
\label{Novikov1}(x\circ y)\circ z&=&(x\circ z)\circ y,\\
\label{Novikov2}(x\circ y)\circ z-x\circ (y\circ z)&=&(y\circ x)\circ z-y\circ (x\circ z)\quad \text{for all}\: x,y,z\in A.
\end{eqnarray}
Let $B$ be a vector space with a binary operation $\diamond$. If $(B, \circ)$ is a Novikov algebra with $a\circ b:=b\diamond a$ for all $a$, $b\in B$, then $(B, \diamond)$ is called a {\bf right Novikov algebra}.
%Unless otherwise specified, all Novikov algebras mentioned in this paper are left Novikov algebras.
Next, we introduce the definition of transposed Novikov-Poisson algebras.}

\delete{\begin{defi}
The triple $(A,\cdot,\circ)$ is called a {\bf \transNovikovpoisson algebra} if $(A,\cdot)$ is a commutative and associative algebra, $(A,\circ)$ is a Novikov algebra, and they satisfy the following compatibility conditions
\begin{eqnarray}
&&\label{transNP1}\quad(x\cdot y)\circ z=(x\cdot z)\circ y,\\
&&\label{transNP2}2z\cdot(x\circ y)=(z\cdot x)\circ y+x\circ(z\cdot y)\quad \text{for all}\: x,y,z\in A.
\end{eqnarray}
\end{defi}

\begin{rmk}
It is easy to see that both a Novikov algebra with the trivial commutative associative algebra and a commutative associative algebra with the trivial Novikov algebra are transposed Novikov-Poisson algebras.  We call such transposed Novikov-Poisson algebras {\bf trivial}.
\end{rmk}}

\delete{Recall \cite{Bai} that the triple \( (L, \cdot, [\cdot,\cdot]) \) is called a {\bf transposed Poisson algebra} if \( (L, \cdot) \) is a commutative associative algebra, \( (L, [\cdot,\cdot]) \) is a Lie algebra and they satisfy the following compatibility condition
\begin{eqnarray}
&&\label{transposedpoissonlu}2z \cdot [x, y] = [z \cdot x, y] + [x, z \cdot y] \quad \text{for all}\: x, y, z \in L.
\end{eqnarray}}

There are close relationships between \transNovikovpoisson algebras and transposed Poisson algebras.
First, we present two direct constructions of transposed Poisson algebras from \transNovikovpoisson algebras
\begin{thm}\label{thmm1}
\label{lidaishu}
Let $(A,\cdot,\circ)$ be a \transNovikovpoisson algebra with a derivation $D$, i.e., $D$ is a derivation of both $(A,\cdot)$ and $(A,\circ)$. Then we obtain the following conclusions.
\begin{enumerate}
\item\label{thm1-1} The triple $(A,\cdot,[\cdot,\cdot])$  is a transposed Poisson algebra, where $[\cdot,\cdot]$ is given by
\begin{eqnarray*}
[x,y]=x\circ y-y\circ x\;\;\;\text{for all $x$, $y\in A$.}
\end{eqnarray*}
\item \label{thm1-2} The triple $(A,\cdot,[\cdot,\cdot])$  is a transposed Poisson algebra, where $[\cdot,\cdot]$ is given by
\begin{eqnarray}
&&\label{jiaohuanzidaozi}[x,y]=D(x)\circ y-D(y)\circ x\:\:\text{for all}\:x,y\in A.
\end{eqnarray}
\end{enumerate}
\end{thm}
\begin{proof}
(\ref{thm1-1}) Since $(A,\circ)$ is a Novikov algebra,  $(A,[\cdot,\cdot])$ is a Lie algebra.
By Eq.~\eqref{transNP2}, for all $x,y,z\in A$, we have
\begin{eqnarray*}
2z\cdot[x,y]&=&2z\cdot(x\circ y-y\circ x),\\
&=&(z\cdot x)\circ y+x \circ (z\cdot y)-(z\cdot y)\circ x-y\circ (z\cdot x)\\
&=&[z\cdot x,y]+[x,z\cdot y].
\end{eqnarray*}
Therefore, $(A,\cdot,[\cdot,\cdot])$ is a transposed Poisson algebra.

(\ref{thm1-2})  It is direct to check that $(A,[\cdot,\cdot])$ is a Lie algebra by Eqs.~\eqref{Novikov1} and ~\eqref{Novikov2}.
By Eqs.~\eqref{transNP1} and ~\eqref{transNP2}, for all $x,y,z\in A$, we obtain
\begin{eqnarray*}
2z\cdot [x,y]&=&2z\cdot(D(x)\circ y -D(y)\circ x)\\
&=&(z\cdot D(x))\circ y+D(x)\circ (z\cdot y)-(z\cdot D(y))\circ x-D(y)\circ (z\cdot x),
\end{eqnarray*}
and
\begin{eqnarray*}
[z\cdot x,y]+[x,z\cdot y]&=&D(z\cdot x)\circ y-D(y)\circ (z\cdot x)+D(x)\circ (z\cdot y)-D(z\cdot y)\circ x\\
&=&(z\cdot D(x))\circ y+D(x)\circ (z\cdot y)-(z\cdot D(y))\circ x-D(y)\circ (z\cdot x)\\
&&\quad+(x\cdot D(z))\circ y-(D(z)\cdot y)\circ x\\
&=&(z\cdot D(x))\circ y+D(x)\circ (z\cdot y)-(z\cdot D(y))\circ x-D(y)\circ (z\cdot x).
\end{eqnarray*}
Therefore, $(A,\cdot,[\cdot,\cdot])$ is a transposed Poisson algebra.
\end{proof}

\begin{rmk}
By \cite[Theorem 3.2]{Bai}, there is a similar construction of transposed Poisson algebras from Novikov-Poisson algebras as in (\ref{thm1-1}) of Theorem \ref{thmm1}. However, the construction of transposed Poisson algebras as in (\ref{thm1-2}) of Theorem \ref{thmm1} doesnot hold in the case of Novikov-Poisson algebras.
\end{rmk}

Next, we show that there is an affinization construction of transposed Poisson algebras from transposed Novikov-Poisson algebras, when ${\bf k}$ is a field of characteristic $0$.

\begin{thm}\label{thmm2}
 Let $A$ be a vector space with two binary operations $\cdot$ and $\circ$, and ${\bf k}$ be a field of characteristic $0$. Define two binary operations $\cdot$ and $[\cdot,\cdot]$ on $A[t,t^{-1}]\coloneqq A\otimes \bfk[t,t^{-1}]$ by
\begin{eqnarray}
&&\label{luolang1}(xt^m)\cdot (yt^n)=(x\cdot y)t^{m+n},\\
&&\label{luolang2}[xt^m, yt^n]=m(x\circ y)t^{m+n-1}-n(y\circ x)t^{m+n-1},
\end{eqnarray}
for all $x,y\in A$ and $m, n\in\mathbb{Z}$, where $x t^m:=x\otimes t^{m}$. Then $(A[t,t^{-1}],\cdot,[\cdot,\cdot])$ is a transposed Poisson algebra if and only if $(A, \cdot,\circ)$ is a {\transNovikovpoisson algebra}.
\end{thm}
\begin{proof}
Obviously, $(A,\cdot)$ is a commutative associative algebra if and only if $(A[t,t^{-1}],\cdot)$ is a commutative associative algebra. By Theorem \ref{thmm0}, $(A,\circ)$ is a Novikov algebra if and only if $(A[t,t^{-1}],[\cdot,\cdot])$ is a Lie algebra.
By Eqs.~\eqref{luolang1} and ~\eqref{luolang2}, we obtain
\begin{eqnarray*}
2(zt^k)\cdot [x t^m, y t^n] &=& 2(zt^k)\cdot\big(m(x\circ y)t^{m + n - 1} - n(y\circ x)t^{m + n - 1}\big)\\
&=& 2m\big(z\cdot (x\circ y)\big)t^{k + m + n - 1} - 2n\big(z\cdot (y\circ x)\big)t^{k + m + n - 1},
\end{eqnarray*}
\begin{eqnarray*}
&&[(zt^k)\cdot (x t^m), y t^n] = [(z\cdot x) t^{m + k}, y t^n] = (m+k)\big((z\cdot x)\circ y\big)t^{m + k + n - 1} - n\big(y\circ (z\cdot x)\big)t^{m + k + n - 1},
\end{eqnarray*}
and
\begin{eqnarray*}
&&[x t^m, (zt^k)\cdot (y t^n)] = [x t^m, (z\cdot y) t^{k + n}] = m\big(x\circ (z\cdot y)\big)t^{m + k + n - 1} - (n+k)\big((z\cdot y)\circ x\big)t^{m + k + n - 1},
\end{eqnarray*}
for all $x,y,z\in A$ and $m,n,l\in\mathbb{Z}$. Comparing the coefficients of $m,n,k$, we obtain that  $(A[t,t^{-1}],\cdot,[\cdot,\cdot])$ satisfies Eq.~\eqref{transposedpoissonlu} if and only if $(A,\cdot,\circ)$ satisfies Eqs.~\eqref{transNP1} and ~\eqref{transNP2}. Then this conclusion holds.
\end{proof}

\begin{ex}\label{ex-dim-1}
Let $(A={\bf k}e, \cdot, \circ)$ be the transposed Novikov-Poisson algebra defined by
\begin{eqnarray}
e\cdot e=\alpha e,\;\;e\circ e=e,
\end{eqnarray}
where $\alpha\in {\bf k}$. Then by Theorem \ref{thmm2}, there is a transposed Poisson algebra structure on $A[t,t^{-1}]$ defined by
\begin{eqnarray}
et^m\cdot et^n=\alpha et^{m+n},\;\;[et^m, et^n]=(m-n)et^{m+n-1}.
\end{eqnarray}
\end{ex}

Finally, we present a generalization of the affinization construction of  transposed Poisson algebras.
\delete{\begin{defi}
 A {\bf right  Novikov algebra} is a vector space $A$ with a binary
 operation $\diamond $ satisfying
    \begin{eqnarray}
 &&\label{rightNovikov1}(a\diamond b)\diamond c-a\diamond (b\diamond c)=(a\diamond c)\diamond b-a\diamond (c\diamond b),\\
&&\label{rightNovikov2}
 a\diamond (b\diamond c)=b\diamond (a\diamond c) ~~~~\tforall  a, b,c\in A.
    \end{eqnarray}
\end{defi}}

\begin{defi}
 A {\bf right differential Novikov-Poisson algebra} is a triple $(B,\cdot,\diamond)$ where $(B,\cdot)$ is a commutative associative algebra, $(B,\diamond)$ is a right Novikov algebra and they satisfy
\begin{eqnarray}
&&\label{rightdiffNovikov1}
(a\diamond b )\cdot c=a\diamond(b\cdot c),\\
&&\label{rightdiffNovikov2} (a\cdot b)\diamond c=(a\diamond c)\cdot b+a\cdot(b\diamond c)\;\;\;\text{for all $a$, $b$, $c\in B$.}
\end{eqnarray}
\end{defi}
\begin{rmk}
Let $A$ be a vector space with two binary operations $\cdot$ and $\circ$. If $(A, \cdot, \diamond)$ is a right differential Novikov-Poisson algebra with $a\diamond b:=b\circ a$ for all $a$, $b\in A$, then $(A, \cdot, \circ)$ \delete{is called a {\bf differential Novikov-Poisson algebra}, which was introduced in \cite{BCZ}.}is a differential Novikov-Poisson algebra \cite{BCZ}.

Let $(B,\cdot,\diamond)$ be a right differential Novikov-Poisson algebra. By Eqs.~\eqref{rightdiffNovikov1} and~\eqref{rightdiffNovikov2}, for all $a$, $b$, $c\in B$, we have
\begin{eqnarray*}
&&(a\cdot b)\diamond c-(a\cdot c)\diamond b=(a\diamond c)\cdot b+a\cdot(b\diamond c)-(a\diamond b)\cdot c-a\cdot (c\diamond b)=a\cdot (b\diamond  c)-a\cdot (c\diamond b).
\end{eqnarray*}
Therefore, we have
\begin{eqnarray}
&&\label{difflem}(a\cdot b)\diamond c-(a\cdot c)\diamond b=a\cdot (b\diamond c)-a\cdot (c\diamond b).
\end{eqnarray}
\end{rmk}

There is a natural construction of right differential Novikov-Poisson algebras from commutative associative algebras with a derivation.

\begin{pro}\label{constr-rdNP}
Let $(B, \cdot)$ be a commutative associative algebra with a derivation $D$. Define a binary operation $\diamond$ on $B$ as follows:
\begin{eqnarray}
a\diamond b:=D(a)\cdot b \;\;\;\text{for all $a$, $b\in B$.}
\end{eqnarray}
Then $(B, \cdot,\diamond)$ is a right differential Novikov-Poisson algebra.
\end{pro}
\begin{proof}
It is straightforward.
\end{proof}
\begin{ex}\label{ex1}
Let $({\bf k}[t,t^{-1}],\cdot)$ be the Laurent polynomial algebra and $D=\frac{d}{dt}$. Then by Proposition \ref{constr-rdNP}, there is a right differential Novikov-Poisson algebra $({\bf k}[t,t^{-1}],\cdot,\diamond)$, where $\diamond$ is defined by
\begin{eqnarray}
t^m\diamond t^n=mt^{m+n-1}\;\;\;\text{for all $m$, $n\in \mathbb{Z}$.}
\end{eqnarray}
\end{ex}
\delete{\begin{lem}
Let $(A,\cdot,\circ)$ be a right differential Novikov-Poisson algebra. Then the following identity holds for all $x,y\in A$
\begin{eqnarray}
&&\label{difflem}(x\cdot y)\circ z-(x\cdot z)\circ y=x\cdot (y\circ z)-x\cdot (z\circ y).
\vspd
\vspa
\end{eqnarray}
\end{lem}

\begin{proof}
By Eqs.~\eqref{rightdiffNovikov1} and~\eqref{rightdiffNovikov2}, we have \vspa
\begin{eqnarray*}
&&(x\cdot y)\circ z-(x\cdot z)\circ y=(x\circ z)\cdot y+x\cdot(y\circ z)-(x\circ y)\cdot z-x\cdot (z\circ y)=x\cdot (y\circ z)-x\cdot (z\circ y).
\vspb
\vspa
\end{eqnarray*}
\quad This completes the proof.
\vspa
\end{proof}}

\begin{pro}\label{pro-constr-trp}
Let $(A, \cdot_1, \circ)$ be a {\transNovikovpoisson algebra} and $(B, \cdot_2, \diamond)$ be a right differential Novikov-Poisson algebra. Define two binary operations $\cdot$, $[\cdot,\cdot]$ on $A\otimes B$ by
\begin{eqnarray}
&&\label{afftensor1}(x\otimes a) \cdot (y \otimes b) = x \cdot_1 y \otimes a \cdot_2 b,\\
&&\label{afftensor2}[x \otimes a, y \otimes b] = x\circ y\otimes a\diamond b-y\circ x \otimes b\diamond a,
\end{eqnarray}
for all $x,y\in A$ and $a,b\in B$. Then $(A\otimes B,\cdot,[\cdot,\cdot])$ is a transposed Poisson algebra.
\end{pro}

\begin{proof}
Obviously, $(A\otimes B,\cdot)$ is a commutative associative algebra. By \cite[Theorem 2.0]{HBG}, $(A\otimes B, [\cdot,\cdot])$ is a Lie algebra.
Next, we need to check Eq.~\eqref{transposedpoissonlu}.  By Eqs.~\eqref{transNP1}, \eqref{transNP2} and \eqref{difflem}, we have
\begin{eqnarray*}
&&2(z \otimes c)\cdot [x \otimes a, y \otimes b]-[(z \otimes c)\cdot(x \otimes a), y \otimes b]-[x \otimes a,(z \otimes c)\cdot(y \otimes b)]\\
&&= 2\big(z\cdot_1(x\circ y)\big) \otimes \big(c\cdot_2(a\diamond b)\big)-2\big(z\cdot_1(y\circ x) \big)\otimes \big(c\cdot_2( b\diamond a)\big)-\big( (x\cdot_1z)\circ y\big)\otimes\big((a\cdot_2c)\diamond b\big)\\
&&+\big( y\circ(x\cdot_1 z)\big)\otimes\big(b\diamond(a\cdot_2 c)\big)-\big( x\circ(y\cdot_1 z)\big)\otimes\big(a\diamond (b\cdot_2 c)\big)+\big( (y\cdot_1 z)\circ x\big)\otimes\big((b\cdot_2 c)\diamond a\big)\\
&&=\big(((x\cdot_1 z)\circ y\big)\otimes \big(c\cdot_2(a\diamond b)\big)-\big((x\cdot_1z)\circ y \big)\otimes \big(c\cdot_2( b\diamond a)\big)\\
&&-\big( (x\cdot_1 z)\circ y\big)\otimes\big((c\cdot_2 a)\diamond b\big)+\big( (x\cdot_1 z)\circ y\big)\otimes\big((c\cdot_2 b)\diamond a\big)\\
&&=0.
\end{eqnarray*}
Then $(A\otimes B,\cdot,[\cdot,\cdot])$ is a transposed Poisson algebra.
\end{proof}

\begin{rmk}
Let $(A, \cdot_1, \circ)$ be a {\transNovikovpoisson algebra} and $(B, \cdot_2, \diamond)$ be the right differential Novikov-Poisson algebra given in Example \ref{ex1}. Then the transposed Poisson algebra $(A\otimes B,\cdot,[\cdot,\cdot])$ constructed in Proposition \ref{pro-constr-trp} is just $(A[t,t^{-1}],\cdot,[\cdot,\cdot])$ given in Theorem \ref{thmm2}.
\end{rmk}

\subsection{Properties of transposed Novikov-Poisson algebras }
\subsubsection{A construction of transposed Novikov-Poisson algebras from commutative associative algebras}
\begin{defi}~\cite{Centroids}
Let $(A, \ast)$ be a commutative associative algebra or a Novikov algebra. The {\bf centroid} of \( (A, \ast) \) is the set of linear maps \( \phi : A \to A \) such that \( \phi(x \ast y) = x \ast \phi(y)=\phi(x)\ast y \) for all \( x, y \in A \). \delete{An algebra is called {\bf central} if for every \( \phi \in \Gamma(A) \) we have that \( \phi(x) = \alpha x \) for some \( \alpha \in \mathbb{F} \) and any \( x \in A \). \par} We denote it by \( \Gamma((A, \ast)) \) .
\end{defi}
\begin{rmk}\label{rmk-centroid}
Let $(A, \cdot)$ be a commutative associative algebra and $p\in A$.
Then $L_\cdot(p)\in \Gamma((A,\cdot))$. Let $\alpha\in {\bf k}$. Define $\widetilde{\alpha}: A\rightarrow A$ by
\begin{eqnarray*}
\widetilde{\alpha}(x)=\alpha x\;\;\; \text{for all $x\in A$.}
\end{eqnarray*}
Then $\widetilde{\alpha}\in \Gamma((A,\cdot))$.

\delete{Then we have
\begin{eqnarray}\label{centroid}
\Gamma((A,\cdot)) = \{ \phi \in \mathrm{End}(A) \mid \phi(x\cdot y) = \phi(x)\cdot y = x\cdot \phi(y) \;\;\text{for all $x, y \in A $.} \}
\end{eqnarray}}
\end{rmk}

\begin{pro}\label{centroids}
Let \((A, \cdot)\) be a commutative associative algebra and $ \phi \in \Gamma((A, \cdot))$. Define a binary operation $\circ$ on $A$ as follows:
\begin{eqnarray}
&&\label{xDy}x \circ y := x \cdot \phi(y) \quad \text{for all } x, y \in A.
\end{eqnarray}
Then \((A, \cdot, \circ)\) is a transposed Novikov-Poisson algebra.
\end{pro}

\begin{proof}
Since \( \phi  \in  \Gamma((A, \cdot)) \), we have
\begin{eqnarray*}
x \circ y = x \cdot \phi(y) = \phi(x \cdot y) = y \cdot \phi(x).
\end{eqnarray*}
Then for all $x,y,z\in A$, we obtain
\begin{eqnarray*}
&&(x \circ y) \circ z = (x \cdot \phi(y)) \cdot \phi(z) = (x \cdot \phi(z)) \cdot \phi(y) = (x \circ z) \circ y,\\
&&x \circ (y \circ z) = x \circ (y \cdot \phi(z)) = x \cdot \phi(y \cdot \phi(z)) = x \cdot \phi(y) \cdot \phi(z),\\
&&(x \circ y) \circ z = x \cdot \phi(y) \cdot \phi(z),
\end{eqnarray*}
and
\begin{eqnarray*}
\quad y \circ (x \circ z) = y \cdot \phi(x) \cdot \phi(z).
\end{eqnarray*}
Thus, we obtain
\begin{eqnarray*}
&&(x \circ y) \circ z - (y \circ x) \circ z = 0 = x \circ (y \circ z) - y \circ (x \circ z),\\
&&2z \cdot (x \circ y) = 2z \cdot x \cdot \phi(y) = z \cdot x \cdot \phi(y) + x \cdot \phi(z \cdot y) = (z \cdot x) \circ y + x \circ (z \cdot y),\end{eqnarray*}
and
\begin{eqnarray*}
&&(x \cdot y) \circ z = x \cdot y \cdot \phi(z) = x \cdot z \cdot \phi(y) = (x \cdot z) \circ y.
\end{eqnarray*}
Therefore, $(A,\cdot,\circ)$ is a {\transNovikovpoisson algebra}.
\end{proof}

\begin{cor}
Let \((A, \cdot)\) be a commutative associative algebra, $p\in A$ or $p\in {\bf k}$.
Define a binary operation $\circ$ on $A$ as follows:
\begin{eqnarray}
&&\label{xDy}x \circ y := p\cdot x \cdot y \quad \text{for all } x, y \in A.
\end{eqnarray}
Then \((A, \cdot, \circ)\) is a transposed Novikov-Poisson algebra.
\delete{\item Define a binary operation $\circ$ on $A$ as follows:
\begin{eqnarray}
&&\label{xDy}x \circ y := \alpha  x \cdot y\quad \text{for all } x, y \in A.
\end{eqnarray}
Then \((A, \cdot, \circ)\) is a transposed Novikov-Poisson algebra.
\end{eqnarray}}
\end{cor}
\begin{proof}
It follows directly by Proposition \ref{centroids} and Remark \ref{rmk-centroid}.
\end{proof}

Finally, we present an example of transposed Novikov-Poisson algebras.
\begin{ex}
Let $A$ be a 3-dimensional vector space over $\mathbb{C}$ with a basis $\{e_1, e_2,e_3\}$. Then $(A, \cdot)$ is a commutative associative algebra with the non-zero products given by
\begin{equation*}
e_1 \cdot e_3 =e_3\cdot e_1= e_1,\quad e_2 \cdot e_3=e_3\cdot e_2 = e_2,\quad e_3\cdot e_3=e_3.
\end{equation*}
By~\cite{Centroids}, the centroid of $(A,\cdot)$ consists of linear maps $\phi$ defined by
\begin{eqnarray*}
&&\phi(e_1)=a_{11}e_1,\quad \phi(e_{2})=a_{11}e_2,\quad \phi(e_3)=a_{13}e_1+a_{23}e_2+a_{11}e_3,
\end{eqnarray*}
where $a_{11}$, $a_{23}$ and $a_{13}\in \mathbb{C}$.
By Proposition~\ref{centroids}, we obtain a Novikov algebra $(A,\circ)$ with the non-zero products given by
\begin{eqnarray*}
e_1\circ e_3=e_3\circ e_1=a_{11}e_1,\quad e_2\circ e_3=e_3\circ e_2=a_{11}e_2,\quad e_3\circ e_3=a_{13}e_1+a_{23}e_2+a_{11}e_3.
\end{eqnarray*}
Then $(A,\cdot,\circ)$ is a {\transNovikovpoisson algebra}.
\end{ex}

\subsubsection{Independence of Novikov-Poisson algebras and  transposed Novikov-Poisson algebras}
\begin{defi}\cite{Xu1}
A {\bf Novikov-Poisson algebra} is a triple $(A, \cdot, \circ)$, where $(A,\cdot)$ is a commutative associative algebra, $(A,\circ)$ is a Novikov algebra,
and they satisfy the following compatibility conditions
\begin{eqnarray}
&&\label{NP1} (x\cdot y) \circ z = x\cdot (y \circ z),\\
&&\label{NP2}  (x \circ y)\cdot z - (y \circ x) \cdot z = x \circ (y\cdot z) - y \circ (x\cdot z)\;\;\;\text{for all $x$, $y$, $z\in A$.}
\end{eqnarray}
\end{defi}

We show that the compatibility relations of the Novikov-Poisson algebra and those of the
transposed Novikov-Poisson algebra are independent in the following sense.
%The following Proposition illustrates the relationship between the {\transNovikovpoisson algebra} and the Novikov algebra.\vspa
\begin{pro}\label{guanxi}
Let \((A, \cdot, \circ)\) be a {\transNovikovpoisson algebra}. Then the following statements are equivalent.\vspa
\begin{enumerate}
\item \label{trans-NP1} The algebra \((A, \cdot, \circ)\)  is a Novikov-Poisson algebra.
\item \label{trans-NP2}The identity \((x\cdot y)\circ z = z\circ (x\cdot y)\) holds for all \(x, y, z \in A\).
\item \label{trans-NP3}\(L_\cdot(x) \in \Gamma((A,\circ))\) for all $ x\in A$\delete{, where \(L_x : A \to A\) is the linear operator of left multiplication given by \(L_x(y) = x \cdot y\)}.
\item \label{trans-NP4}\(L_\circ(x) \in \Gamma((A,\cdot))\)  for all $x \in A$.
\end{enumerate}
\end{pro}

\begin{proof}
$\eqref{trans-NP1} \Rightarrow \eqref{trans-NP2} $:\:For all $x,y,z\in A$, by Eqs.~\eqref{transNP1},~\eqref{transNP2} and ~\eqref{NP1}, we obtain
\begin{eqnarray*}
 (x\cdot y) \circ z-z \circ (x\cdot y) &=& (x\cdot z) \circ y- z \circ (x\cdot y)\\
&=& 2x\cdot (z \circ y) - 2z \circ (x\cdot y) \\
&=&2(x\cdot z) \circ y - 2z \circ (x\cdot y).
\end{eqnarray*}
Thus, we obtain $(x\cdot y) \circ z-z \circ (x\cdot y)=0$.

$\eqref{trans-NP2}\Rightarrow\eqref{trans-NP3}$:\: By \eqref{trans-NP2}, for all $x,y,z\in A$, we obtain
\begin{eqnarray*}
&&z\cdot (x\circ y)=\frac{1}{2}(z\cdot x) \circ y + \frac{1}{2}x \circ (y\cdot z) = \frac{1}{2}(z\cdot x) \circ y + \frac{1}{2}(y\cdot z) \circ x = (z\cdot x) \circ y,\end{eqnarray*}
and
\begin{eqnarray*}
&&z \cdot (x \circ y) = \frac{1}{2}(z\cdot x) \circ y + \frac{1}{2}x \circ (z\cdot y)=\frac{1}{2}(z\cdot y) \circ x + \frac{1}{2}x \circ (z\cdot y) = x \circ (z\cdot y).
\end{eqnarray*}
Therefore, we have $L_\cdot(z) \in \Gamma((A,\circ))$ for all $z\in A$.

$\eqref{trans-NP3}\Rightarrow\eqref{trans-NP1}$: For all $x,y,z\in A$, by the definition of centroid, we have
\begin{align*}
x\cdot(y \circ z)&\!=\!(x\cdot y) \circ z ,\\
(x \circ y)\cdot z - (y \circ x)\cdot z&\!=\!x \circ (y\cdot z) - y \circ (x\cdot z).
\end{align*}
Therefore, $(A,\cdot, \circ)$ is a Novikov-Poisson algebra.

$\eqref{trans-NP4}\Leftrightarrow\eqref{trans-NP2}$: For all $x,y,z\in A$, by the definition of the {\transNovikovpoisson algebra},we have
\begin{eqnarray*}
&&x\cdot(z\circ y)-(z\circ x)\cdot y=\frac{1}{2}(x\cdot z)\circ y+\frac{1}{2}z\circ (x\cdot y)-\frac{1}{2}(y\cdot z)\circ x-\frac{1}{2}z\circ (x\cdot y)=0,
\end{eqnarray*}
and
\begin{eqnarray*}
&&z\circ (x\cdot y)-(z\circ x)\cdot y=z\circ (x\cdot y)-\frac{1}{2}(z\cdot y)\circ x-\frac{1}{2}z\circ (x\cdot y)=\frac{1}{2}(z\circ (x\cdot y)-(x\cdot y)\circ z).
\vspc
\end{eqnarray*}
Therefore, we get $\eqref{trans-NP4}\Leftrightarrow\eqref{trans-NP2}$.
\end{proof}

\begin{cor}
Let $(A,\cdot,\circ)$ be a {\transNovikovpoisson algebra}. If $(A,\cdot)$ has a unit $e$, then $(A,\circ)$ is a commutative Novikov algebra, and hence $(A,\cdot,\circ)$ is a Novikov-Poisson algebra.
\end{cor}

\begin{proof}
The commutativity of $(A,\circ)$ follows from the computation:\vspa
\begin{eqnarray*}
&&x\circ y = (x\cdot e)\circ y = (y\cdot e)\circ x = y\circ x.
\vspa
\end{eqnarray*}
This implies that $(A,\cdot,\circ)$ satisfies \eqref{trans-NP2} in  Proposition \ref{guanxi}. Therefore, $(A,\cdot,\circ)$ is a Novikov-Poisson algebra.
\end{proof}

However, the converse is not true, a {\transNovikovpoisson algebra} that is a Novikov-Poisson algebra is not necessarily a commutative Novikov algebra. A counterexample is provided by the following three-dimensional {\transNovikovpoisson algebra}.\vspa

\begin{ex}
Let $A$ be a 3-dimensional vector space over $\mathbb{C}$ with a basis $\{e_1, e_2,e_3\}$. By \cite{erwei}, $(A, \circ)$ is a Novikov algebra with non-zero products given by
\begin{equation*}
 e_2 \circ e_3 = e_1, \quad e_3 \circ e_2 = -e_1.
\end{equation*}
 Define a commutative associative algebra \((A, \cdot)\), whose
 non-zero products are given by
 \begin{eqnarray*}
&&e_2\cdot e_2 = ne_1,\quad e_2\cdot e_3=le_1=e_3\cdot e_2,\quad  e_3\cdot e_3=ke_1.
 \end{eqnarray*}
where $n,\:l,\:k\in \mathbb{C}$. By a straightforward computation,  $(A,\cdot,\circ)$ is both a {\transNovikovpoisson algebra} and a Novikov-Poisson algebra. However, $(A, \circ)$ is not a commutative Novikov algebra.
\end{ex}

\subsubsection{Identities in {\transNovikovpoisson algebras} and classification of 2-dimensional {\transNovikovpoisson algebras} over $\mathbb{C}$}
\begin{pro}
Let $(A,\cdot,\circ)$ be a {\transNovikovpoisson algebra}. Then the following identities hold for all $x,y,z,h\in A$:
\begin{eqnarray}
\label{Tid1}(x\circ y)\cdot z&=&(x\circ z)\cdot y,\\
\label{Tid2}(x\circ y)\circ (h\cdot z)&=&(x\circ z)\circ (h\cdot y),\\
\label{Tid3}(x\circ y)\circ (h\cdot z)-(y\circ x)\circ (h\cdot z)&=&(h\cdot x)\circ (y\circ z)-(h\cdot y)\circ (x\circ z).
\end{eqnarray}
Moreover, if the characteristic of ${\bf k}$ is not equal to $3$, then the following identity also holds for all $x,y,z,h\in A$:
\begin{eqnarray}
\label{Tid4}(x\circ y)\circ (h\cdot z)&=&(y\circ x)\circ (h\cdot z).
\end{eqnarray}
\end{pro}

\begin{proof}
For all $x,y,z\in A$, by Eqs.~\eqref{transNP1} and~\eqref{transNP2}, we have \vspa
\begin{eqnarray*}
&&(x\circ y)\cdot z=\frac{1}{2}(z\cdot x)\circ y+\frac{1}{2}x\circ (z\cdot y)=\frac{1}{2}(x\cdot y)\circ z+\frac{1}{2}x\circ (z\cdot y)=y\cdot (x\circ z).
\end{eqnarray*}
Hence, Eq.~\eqref{Tid1} holds.

Note that $L_\cdot (h)$ is a $\frac{1}{2}$-derivation of $(A,\circ)$. Therefore, Eqs.~\eqref{Tid2} and~\eqref{Tid3} are special cases of Eqs.~\eqref{1/2id1} and \eqref{1/2id2} in Lemma~\ref{1/2id}, respectively. For the details, one can refer to the proof of Lemma~\ref{1/2id}.

By Eqs.~\eqref{Novikov1}, ~\eqref{transNP1},~\eqref{transNP2} and ~\eqref{Tid3}, we obtain
\begin{eqnarray*}
0&=&(h\cdot(y\circ z))\circ x-(h\cdot(x\circ z))\circ y-(x\circ y)\circ (h\cdot z)+(y\circ x)\circ (h\cdot z)\\
&=&\frac{1}{2}((h\cdot y)\circ z)\circ x+\frac{1}{2}(y\circ(h\cdot z))\circ x-\frac{1}{2}((h\cdot x)\circ z)\circ y-\frac{1}{2}(x\circ (h\cdot z))\circ y\\
&&\quad-(x\circ y)\circ (h\cdot z)+(y\circ x)\circ (h\cdot z)\\
&=&\frac{3}{2}(y\circ x)\circ (h\cdot z)-\frac{3}{2}(x\circ y)\circ (h\cdot z).
\end{eqnarray*}
Therefore, if the characteristic of ${\bf k}$ is not equal to $3$, Eq.~\eqref{Tid4} holds.
\end{proof}

For classifying {\transNovikovpoisson} algebras, we can consider compatible commutative associative algebra structures on the known Novikov algebras.

Let $\{e_1,\ldots, e_n\}$ be a basis of a {\transNovikovpoisson algebra} \( (A,\cdot,\circ) \). Obviously, $\cdot$ and $\circ$ are determined by the following characteristic matrices
\begin{eqnarray*}
(C_{ij})= \begin{pmatrix}\sum_{k=1}^n c_{11}^k e_k & \ldots & \sum_{k=1}^n c_{1n}^k e_k \\
\ldots & \ldots & \ldots \\
\sum_{k=1}^n c_{n1}^k e_k & \ldots & \sum_{k=1}^n c_{nn}^k e_k
\end{pmatrix},
(D_{ij})= \begin{pmatrix}
\sum_{k=1}^n d_{11}^k e_k & \ldots & \sum_{k=1}^n d_{1n}^k e_k \\
\ldots & \ldots & \ldots \\
\sum_{k=1}^n d_{n1}^k e_k & \ldots & \sum_{k=1}^n d_{nn}^k e_k
\end{pmatrix},
\end{eqnarray*}
where \( e_i \cdot e_j = \sum_{k=1}^n c_{ij}^k e_k \) and \( e_i \circ e_j = \sum_{k=1}^n d_{ij}^k e_k \).

Suppose that $(A, \circ)$ is a known Novikov algebra. For characterizing compatible commutative associative algebra structures on $(A, \circ)$,  by the definition of transposed Novikov-Poisson algebras, we need to solve the following equations
\begin{eqnarray}
&&\label{erwei1}c_{ij}^p = c_{ji}^p, \quad\sum_{l=1}^n c_{ij}^l c_{lk}^p = \sum_{l=1}^n c_{jk}^l c_{il}^p,\\
&&\label{erwei2}\sum_{l = 1}^{n} c_{ij}^{l} d_{lk}^{p} = \sum_{l = 1}^{n} c_{ik}^{l} d_{lj}^{p},\, p = 1, \ldots, n;\\
&&\label{erwei3}2\sum_{l = 1}^{n} d_{ij}^{l} c_{kl}^{p} = \sum_{l = 1}^{n} c_{ki}^{l} d_{lj}^{p} + \sum_{l = 1}^{n} c_{kj}^{l} d_{il}^{p},\, p = 1, \ldots, n.
\end{eqnarray}

Note that the classification of 2-dimensional Novikov algebras over $\mathbb{C}$ up to isomorphism was given in ~\cite{erwei}. Based on this classification result, we give a classification of 2-dimensional transposed Novikov-Poisson algebras up to isomorphism.
\begin{pro}\label{2-dim}
Let $(A, \cdot, \circ)$ be a 2-dimensional transposed Novikov-Poisson algebra over $\mathbb{C}$ with a non-zero binary operation $\circ$. Then $(A, \cdot, \circ)$ is isomorphic to one of the following cases:
\begin{table}[h]
\begin{tabular}{|c|c|}
\hline
Characteristic matrix of $(A, \circ)$  & Compatible characteristic matrix of $(A, \cdot)$ \\
\hline
$(T2)$ $\begin{pmatrix} 0 & 0 \\ 0 & e_1 \end{pmatrix}$ & $\begin{pmatrix} 0 & me_1 \\ me_1 & ne_1+me_2  \end{pmatrix}$ \\
\hline
$(T3)$ $\begin{pmatrix} 0 & 0 \\ -e_1 & 0 \end{pmatrix}$ & $\begin{pmatrix} 0 & 0 \\ 0 & 0 \end{pmatrix}$ \\
\hline
$(N1)$ $\begin{pmatrix} e_1 & 0 \\ 0 & e_2 \end{pmatrix}$ & $\begin{pmatrix} ne_1 & 0 \\ 0 & me_2 \end{pmatrix}$ \\
\hline
$(N2)$ $\begin{pmatrix} 0 & 0 \\ 0 & e_2 \end{pmatrix}$ &  $\begin{pmatrix} ne_1 & 0 \\ 0 & me_2 \end{pmatrix}$ \\
\hline
$(N3)$ $\begin{pmatrix} 0 & e_1 \\ e_1 & e_2 \end{pmatrix}$  & $\begin{pmatrix} 0 & ne_1 \\ ne_1 & me_1+ne_2 \end{pmatrix}$ \\
\hline
$(N4)$ $\begin{pmatrix} 0 & e_1 \\ 0 & e_2 \end{pmatrix}$ & $\begin{pmatrix} 0 & 0 \\ 0 & 0 \end{pmatrix}$ \\
\hline
$(N5)$ $\begin{pmatrix} 0 & e_1 \\ 0 & e_1 + e_2 \end{pmatrix}$  & $\begin{pmatrix} 0 &0 \\ 0 & 0 \end{pmatrix}$ \\
\hline
$(N6)$ $\begin{pmatrix} 0 & e_1 \\ le_1 & e_2 \end{pmatrix},l\neq0,1$ & $\begin{pmatrix} 0 & 0 \\ 0 & 0 \end{pmatrix}$ \\
\hline
\end{tabular}
\end{table}
\end{pro}
\begin{proof}
By the classification result in \cite{erwei}, 2-dimensional non-trivial Novikov algebras up to isomorphism are of the form $(T2)-(T3)$ and $(N1)-(N6)$. Then we can obtain compatible commutative associative algebra structures on these Novikov algebras case by case by direct computations.
\end{proof}
\begin{rmk}
\delete{We observe that all 2-dimensional {\transNovikovpoisson algebras} over $\mathbb{C}$ with the non-trivial commutative associative algebra structure up to isomorphism are commutative Novikov algebras. By using the classification of centroids of 2-dimensional associative algebras over $\mathbb{C}$~\cite{Centroids}, it can be shown that every such 2-dimensional {\transNovikovpoisson algebra} over $\mathbb{C}$ is obtained from some commutative associative algebra by the construction given in Proposition \ref{centroids}.}
The affinization of two-dimensional Novikov algebras produces kinds of infinite-dimensional Lie algebras~\cite{Virasoro3}. Similarly, we can also obtain many infinite-dimensional transposed Poisson algebra structures from the classification of 2-dimensional transposed Novikov-Poisson algebras via the affinization. \delete{So we can derive the corresponding transposed Poisson algebra structures on these Lie algebras.} For example, the Lie algebra $L(N_3)$ via the affinization of the Novikov algebra $N_3$ which is isomorphic to the semi-direct sum of the Witt algebra and its adjoint module, i.e., the centerless $W(2,2)$~\cite{W22}. By Proposition \ref{2-dim}, there are non-trivial transposed Poisson algebra structures on $L(N_3)$, which coincides with \cite[Theorem 22]{1-2}.
\end{rmk}

\subsubsection{Tensor products}

%\subsection{Related structures and tensor products of {\transNovikovpoisson algebra}}
An important property of {\transNovikovpoisson algebras} is that they are closed under taking tensor products.

\begin{pro}\label{tensoris}
Let $(A_1, \cdot_1, \circ_1)$ and $(A_2, \cdot_2, \circ_2)$ be two {\transNovikovpoisson algebras}. Define two binary operations $\cdot$ and $\circ$ on $A_1 \otimes A_2$ by
\begin{eqnarray}
&\label{tensor1}(a_1 \otimes a_2) \cdot (b_1 \otimes b_2) = a_1 \cdot_1 b_1 \otimes a_2 \cdot_2 b_2,\\
&\label{tensor2}(a_1 \otimes a_2) \circ (b_1 \otimes b_2) = a_1 \cdot_1 b_1 \otimes a_2 \circ_2 b_2 + a_1 \circ_1 b_1 \otimes a_2 \cdot_2 b_2,
\end{eqnarray}
for all $a_1,b_1\in A_1$ and $a_2,b_2\in A_2$. Then $(A_1\otimes A_2,\cdot,\circ)$ is a {\transNovikovpoisson algebra}.
\end{pro}

\begin{proof}
\delete{For notational simplicity, we suppress the subscripts of the commutative and associative multiplication operations $\cdot_1$ in $A_1$ and $\cdot_2$ in $A_2$ throughout this proof. Parentheses in products are omitted due to associativity.\par}
It is obvious that  $(A_1\otimes A_2,\cdot)$ is a commutative associative algebra. Since $(A_1,\cdot_1)$ and $(A_2,\cdot_2)$ are commutative associative algebras, then for all $a_1,b_1,c_1\in A_1$ and $a_2,b_2,c_2\in A_2$, we have
\begin{eqnarray*}
&&\left( (a_1 \otimes a_2) \circ (b_1 \otimes b_2) \right) \circ (c_1 \otimes c_2) \\
&&\quad= \left( (a_1\cdot_1b_1) \otimes (a_2 \circ_2 b_2) + (a_1 \circ_1 b_1) \otimes (a_2 \cdot_2b_2) \right) \circ (c_1 \otimes c_2)\\
&&\quad= (a_1 \cdot_1b_1 \cdot_1c_1) \otimes (a_2 \circ_2 b_2) \circ_2 c_2 + \left( (a_1\cdot_1 b_1) \circ_1 c_1 \right) \otimes \left( (a_2 \circ_2 b_2) \cdot_2c_2 \right) \\
&&\qquad+ \left( (a_1 \circ_1 b_1) \cdot_1 c_1 \right) \otimes \left( (a_2\cdot_2 b_2) \circ c_2 \right) + \left( (a_1 \circ_1 b_1) \circ_1 c_1 \right) \otimes (a_2\cdot_2 b_2 \cdot_2c_2).\end{eqnarray*}
Therefore, we obtain
\begin{eqnarray*}
&&\left( (b_1 \otimes b_2) \circ (a_1 \otimes a_2) \right) \circ (c_1 \otimes c_2)\\
&&\quad= (a_1\cdot_1 b_1\cdot_1 c_1) \otimes (b_2 \circ_2 a_2) \circ_2 c_2 + \left( (a_1\cdot_1 b_1) \circ_1 c_1 \right) \otimes \left( (b_2 \circ_2 a_2) \cdot_2 c_2 \right)\\
&&\qquad+ \left( (b_1 \circ_1 a_1) \cdot_1c_1 \right) \otimes \left( (a_2\cdot_2 b_2) \circ c_2 \right) + \left( (b_1 \circ_1 a_1) \circ_1 c_1 \right) \otimes (a_2\cdot_2 b_2\cdot_2 c_2).
\end{eqnarray*}
Taking the difference between two equalities above leads to
\begin{eqnarray*}
&&\left( (a_1 \otimes a_2) \circ (b_1 \otimes b_2) \right) \circ (c_1 \otimes c_2)-\left( (b_1 \otimes b_2) \circ (a_1 \otimes a_2) \right) \circ (c_1 \otimes c_2)=(B_1)+(B_2)+(B_3),
\end{eqnarray*}
where
\begin{eqnarray*}
&&(B_1)=(a_1\cdot_1b_1 \cdot_1c_1) \otimes \left( (a_2 \circ_2 b_2 - b_2 \circ_2 a_2) \circ_2 c_2 \right),\\
&&(B_2)=\left( (a_1\cdot_1 b_1) \circ_1 c_1 \right) \otimes \left( c_2 \cdot_2(a_2 \circ_2 b_2 - b_2 \circ_2 a_2) \right) \\
&&\qquad\quad+ \left( c_1\cdot_1 (a_1 \circ_1 b_1 - b_1 \circ_1 a_1) \right) \otimes \left( (a_2\cdot_2 b_2) \circ_2 c_2 \right)\\
&&(B_3)=\left( (a_1 \circ_1 b_1 - b_1 \circ_1 a_1) \circ_1 c_1 \right) \otimes (a_2 \cdot_2b_2\cdot_2 c_2) .
\end{eqnarray*}
By Eqs.~\eqref{tensor1} and ~\eqref{tensor2}, we have
\begin{eqnarray*}
&&(a_1 \otimes a_2) \circ \left( (b_1 \otimes b_2) \circ (c_1 \otimes c_2) \right) \\
&&\quad= (a_1 \otimes a_2) \circ \left( (b_1\cdot_1 c_1) \otimes (b_2 \circ_2 c_2) + (b_1 \circ_1 c_1) \otimes (b_2\cdot_2 c_2) \right)\\
&&\quad= (a_1\cdot_1 b_1\cdot_1 c_1) \otimes \left( a_2 \circ_2 (b_2 \circ_2 c_2) \right) + \left( a_1 \circ_1 (b_1\cdot_1 c_1) \right) \otimes \left( a_2\cdot_2 (b_2 \circ_2 c_2) \right) \\
&&\qquad+ \left( a_1 \cdot_1(b_1 \circ_1 c_1) \right) \otimes \left( a_2 \circ_2 (b_2\cdot_2 c_2) \right) + \left( a_1 \circ_1 (b_1 \circ_1 c_1) \right) \otimes (a_2 \cdot_2b_2 \cdot_2c_2).
\end{eqnarray*}
Therefore, we obtain
\begin{eqnarray*}
&&(b_1 \otimes b_2) \circ \left( (a_1 \otimes a_2) \circ (c_1 \otimes c_2) \right) \\
&&\quad= (a_1 \cdot_1b_1 \cdot_1c_1) \otimes \left( b_2 \circ_2 (a_2 \circ_2 c_2) \right) + \left( b_1 \circ_1 (a_1\cdot_1 c_1) \right) \otimes \left( b_2 \cdot_2(a_2 \circ_2 c_2) \right) \\
&&\qquad+ \left( b_1 \cdot_1(a_1 \circ_1 c_1) \right) \otimes \left( b_2 \circ_2 (a_2 \cdot_2c_2) \right) + \left( b_1 \circ_1 (a_1 \circ_1 c_1) \right) \otimes (a_2\cdot_2 b_2\cdot_2 c_2).
\end{eqnarray*}
Taking the difference between two equations above leads to
\begin{eqnarray*}
 &&(a_1 \otimes a_2) \circ \left( (b_1 \otimes b_2) \circ (c_1 \otimes c_2) \right)-(b_1 \otimes a_2) \circ \left( (a_1 \otimes b_2) \circ (c_1 \otimes c_2) \right)=( C_1) +( C_2) +( C_3),
\vspc
\end{eqnarray*}
where
\begin{eqnarray*}
&& (C_1) = (a_1 \cdot_1b_1\cdot_1 c_1) \otimes \left( a_2 \circ_2 (b_2 \circ_2 c_2) - b_2 \circ_2 (a_2 \circ_2 c_2) \right),\\
&&(C_2 )= \left( a_1 \circ_1 (b_1 \cdot_1c_1) \right) \otimes \left( a_2 \cdot_2(b_2 \circ_2 c_2) \right) + \left( a_1\cdot_1 (b_1 \circ_1 c_1) \right) \otimes \left( a_2 \circ_2 (b_2\cdot_2 c_2) \right) \\
&&- \left( b_1 \circ_1 (a_1\cdot_1 c_1) \right) \otimes \left( b_2 \cdot_2(a_2 \circ_2 c_2) \right) - \left( b_1 \cdot_1(a_1 \circ_1 c_1) \right) \otimes \left( b_2 \circ_2 (a_2\cdot_2 c_2) \right),\\
&&(C_3) = \left( a_1 \circ_1 (b_1 \circ_1 c_1) - (b_1 \circ_1 (a_1 \circ_1 c_1) \right) \otimes (a_2\cdot_2 b_2 \cdot_2c_2).
\end{eqnarray*}
Since $(A_1, \circ_1)$ and $(A_2, \circ_2)$ are Novikov algebras,  we have $(B_1) =( C_1)$, $(B_3) = (C_3)$.
Furthermore, by Eq.~\eqref{transNP2}, we obtain
\begin{eqnarray*}
(B_2) &=& \frac{1}{2} \left( (a_1 \cdot_1b_1) \circ_1 c_1 \right) \otimes \left( (c_2 \cdot_2a_2) \circ_2 b_2 + a_2 \circ_2 (c_2 \cdot_2b_2) \right) \\
&&\quad- \frac{1}{2} \left( (a_1\cdot_1 b_1) \circ_1 c_1 \right) \otimes \left( (b_2\cdot_2 c_2) \circ_2 a_2 + b_2 \circ_2 (c_2\cdot_2 a_2) \right)\\
&&\quad+ \frac{1}{2} \left( (c_1 \cdot_1a_1) \circ_1 b_1 + a_1 \circ_1 (b_1\cdot_1 c_1) \right) \circ \left( (a_2 \cdot_2b_2) \circ_2 c_2 \right) \\
&&\quad- \frac{1}{2} \left( b_1 \circ_1 (a_1 \cdot_1c_1) + (c_1\cdot_1 b_1) \circ_1 a_1 \right) \otimes \left( (a_2\cdot_2 b_2) \circ_2 c_2 \right),
\end{eqnarray*}
and
\begin{eqnarray*}
(C_2)&=&\frac{1}{2} \left( a_1 \circ_1 (b_1 \cdot_1c_1) \right) \otimes \left( (a_2\cdot_2 b_2) \circ_2 c_2 + b_2 \circ_2 (a_2 \cdot_2c_2) \right) \\
&&\quad+ \frac{1}{2} \left( (a_1 \cdot_1b_1) \circ_1 c_1 + b_1 \circ_1 (a_1\cdot_1 c_1) \right) \otimes \left( a_2 \circ_2 (b_2 \cdot_2c_2) \right) \\
&&\quad- \frac{1}{2} \left( b_1 \circ_1 (a_1\cdot_1 c_1) \right) \otimes \left( (a_2 \cdot_2b_2) \circ_2 c_2 + a_2 \circ_2 (b_2\cdot_2 c_2) \right)\\
&&\quad- \frac{1}{2} \left( (a_1\cdot_1 b_1) \circ_1 c_1 + a_1 \circ_1 (b_1\cdot_1 c_1) \right) \otimes \left( b_2 \circ_2 (a_2 \cdot_2c_2) \right)\\
&=& \frac{1}{2} \left( a_1 \circ_1 (b_1\cdot_1 c_1) \right) \otimes \left( (a_2\cdot_2 b_2) \circ_2 c_2 \right) + \frac{1}{2} \left( (a_1\cdot_1 b_1) \circ_1 c_1 \right) \otimes \left( a_2 \circ_2 (b_2\cdot_2 c_2) \right)\\
&&- \frac{1}{2} \left( b_1 \circ_1 (a_1\cdot_1 c_1) \right) \otimes \left( (a_2\cdot_2 b_2) \circ_2 c_2 \right) - \frac{1}{2} \left( (a_1 \cdot_1b_1) \circ_1 c_1 \right) \otimes \left( b_2 \circ_2 (a_2\cdot_2 c_2) \right) .
\end{eqnarray*}
Since $(a_1\cdot_1b_1)\circ_1c_1=(a_1\cdot_1c_1)\circ_1b_1$ and $(a_2\cdot_2b_2)\circ_2c_2=(a_2\cdot_2c_2)\circ_2b_2$, we have $(B_2)=(C_2)$. Therefore, Eq. (\ref{Novikov1}) holds.

Since
\begin{eqnarray*}
&&\big((a_1 \otimes a_2) \circ (b_1 \otimes b_2)\big) \circ (c_1 \otimes c_2)\\
&&\quad= \Big(\big(a_1\cdot_1b_1 \otimes (a_2 \circ_2 b_2)\big) + \big((a_1 \circ_1 b_1) \otimes a_2\cdot_2b_2\big)\Big) \circ (c_1 \otimes c_2) \\
&&\quad= (a_1\cdot_1b_1\cdot_1c_1) \otimes \big((a_2 \circ_2 b_2) \circ c_2\big) + \big((a_1 \cdot_1 b_1)\circ_1c_1\big) \otimes \big((a_2\circ_2b_2) \cdot_2 c_2\big)\\
&&\qquad+ \big((a_1 \circ_1 b_1) \circ c_1\big) \otimes (a_2\cdot_2b_2\cdot_2c_2) + \big((a_1 \circ_1 b_1)\cdot_1c_1\big) \otimes \big((a_2\cdot_2b_2) \circ_2 c_2\big) \\
&&\quad= (a_1\cdot_1b_1\cdot_1c_1) \otimes \big((a_2 \circ_2 c_2) \circ_2 b_2\big) + \big((a_1 \cdot_1 c_1) \circ_1 b_1\big) \otimes \big((a_2 \circ_2 c_2) \cdot_2 b_2\big)\\
&&\quad+ \big((a_1 \circ_1 c_1) \circ_1 b_1\big) \otimes (a_2\cdot_2b_2\cdot_2c_2) + \big((a_1 \circ_1 c_1)\cdot_1b_1\big) \otimes \big((a_2\cdot_2c_2) \circ_2 b_2\big) \\
&&\qquad= \Big((a_1 \otimes a_2) \circ (c_1 \otimes c_2)\Big) \circ (b_1 \otimes b_2),
\end{eqnarray*}
 $(A_1\otimes A_2,\cdot)$ is a Novikov algebra.

Since
\begin{eqnarray*}
&&2(c_1 \otimes c_2) \cdot \big((a_1 \otimes a_2) \circ (b_1 \otimes b_2)\big) \\
&&\quad= 2(c_1 \otimes c_2)\cdot \Big( a_1\cdot_1b_1 \otimes (a_2 \circ_2 b_2) + (a_1 \circ_1 b_1) \otimes (a_2\cdot_2b_2) \Big) \\
&&\quad=2(a_1\cdot_1b_1\cdot_1c_1) \otimes c_2\cdot_2(a_2 \circ_2 b_2) + 2\big(c_1\cdot_1(a_1 \circ_1 b_1)\big) \otimes (a_2\cdot_2b_2\cdot_2c_2)\\
&&\quad=(a_1\cdot_1b_1\cdot_1c_1)\otimes \big((c_2\cdot_2 a_2)\circ_2 b_2+a_2\circ_2(c_2\cdot_2 b_2)\big)\\
&&\qquad+\big((c_1\cdot_1 a_1)\circ_1 b_1+a_1\circ_1(c_1\cdot_1 b_1)\big)\otimes (a_2\cdot_2b_2\cdot_2c_2)\\
&&\quad=\big((a_1\cdot_1c_1 )\otimes (a_2\cdot_2c_2)\big) \circ (b_1 \otimes b_2)+(a_1 \otimes a_2) \circ \big((b_1\cdot_1c_1) \otimes (b_2\cdot_2c_2)\big) \\
&&\quad= \big((c_1 \otimes c_2)\cdot (a_1 \otimes a_2)\big) \circ (b_1 \otimes b_2)+(a_1 \otimes a_2) \circ \big((c_1 \otimes c_2)\cdot (b_1 \otimes b_2)\big),
\end{eqnarray*}
Eq.~\eqref{transNP2} holds.

Moreover, we have
\begin{eqnarray*}
&& \Big((a_1 \otimes a_2) \cdot (b_1 \otimes b_2)\Big) \circ (c_1 \otimes c_2)\\
 & &\quad= \big(a_1\cdot_1b_1 \otimes a_2\cdot_2b_2)\big) \circ (c_1 \otimes c_2) \\
 & &\quad= (a_1\cdot_1b_1\cdot_1c_1) \otimes \big((a_2\cdot_2b_2) \circ_2 c_2\big) +( \big(a_1\cdot_1b_1)\circ_1c_1) \otimes (a_2\cdot_2b_2\cdot_2c_2\big) \\
 & &\quad=(a_1\cdot_1c_1\cdot_1b_1) \otimes \big((a_2\cdot_2c_2)\circ_2 b_2\big) + \big((a_1\cdot_1c_1) \circ_1 b_1\big) \otimes (a_2\cdot_2b_2\cdot_2c_2) \\
 & &\quad= \Big((a_1 \otimes a_2) \cdot (c_1 \otimes c_2)\Big) \circ (b_1 \otimes b_2).
\end{eqnarray*}
 Therefore, $(A_1\otimes A_2, \cdot, \circ)$ is a transposed Novikov-Poisson algebra.
\end{proof}

\subsubsection{Constructions from  known transposed Novikov-Poisson algebras}
\delete{\begin{defi}~\cite{1-2}
Let $( A, \cdot )$ be an arbitrary associative algebra, and let $p$ and $q$ be two fixed elements of $A$. Then a new algebra is derived from $A$ by using the same vector space structure of $A$ but defining a new multiplication\vspb
\begin{eqnarray*}
x * y = x \cdot p \cdot y - y \cdot q \cdot x,\quad \text{for all}\:x,y\in A.
\vspb
\end{eqnarray*}
The resulting algebra is denoted by $A(p, q)$ and is called the $(p, q)$-mutation of the algebra $A$. In the commutative case, all $(p, q)$-mutations can be reduced to the case $q = 0$. More detail in ~\cite{tubiancankao}.
\vspa
\end{defi}

\begin{pro}
Let $(A, \cdot, \circ)$ be a {\transNovikovpoisson algebra}. Then every mutation $(A, \cdot_p)$ of $(A, \cdot)$ gives a {\transNovikovpoisson algebra} with the same Novikov multiplication.
\vspb
\end{pro}

\begin{proof}
Obviously, every mutation of an associative and commutative algebra gives an associative and commutative algebra. Since\vspa
\begin{eqnarray*}
&&2z \cdot_p (x\circ y) = 2 (z \cdot p) \cdot  (x\circ y)  = (z \cdot p\cdot x)\circ y + x \circ (z \cdot p \cdot y)= (z \cdot_p\cdot x)\circ y + x \circ (z \cdot_p \cdot y),\\
&& (x \cdot_p\cdot y)\circ z=(x \cdot p\cdot y)\circ z=(x \cdot p\cdot z)\circ y=(x \cdot_p\cdot z)\circ y.
\vspb
\end{eqnarray*}
Hence, $(A, \cdot_p, \circ)$ is a {\transNovikovpoisson algebra}.
\end{proof}

A deformation theory of Novikov algebras has been shown to be closely related to Novikov??¨¬CPoisson algebras ~\cite{someresult}. An similar connection also exists with {\transNovikovpoisson algebras}.
Let $(A, \circ)$ be a Novikov algebra, and $g_p : A \times A \to A$ be a binary product defined by\vspb
\begin{eqnarray}
g_q(a, b) = a * b + q G_1(a, b) + q^2 G_2(a, b) + q^3 G_3(a, b)+ \dots,
\end{eqnarray}
where $G_i$ are bilinear products with $G_0(a, b) = a \circ b$. We call $(A_q, g_q)$ a deformation of $(A, \circ)$ if $(A_q, g_q)$ is a family of Novikov algebras. In particular, we call $G_1$ a global deformation if the deformation is given by\vspb
\begin{eqnarray*}
g_q(a, b) = a * b + q G_1(a, b),
\end{eqnarray*}
And $G_1$ is a global deformation if and only if
\begin{eqnarray}
&&\label{G1}G_1(a, b\circ c) - G_1(a \circ b, c) + G_1(b \circ a, c) - G_1(b, a \circ c)\\
&&+ a \circ G_1(b, c) - G_1(a, b) \circ c + G_1(b, a) \circ c - b \circ G_1(a, c) = 0,\notag \\
&&\label{G2}G_1(a, b \circ c) - G_1(a \circ b, c) + G_1(b \circ a, c) - G_1(b, a \circ c) \\
&&+ a \circ G_1(b, c)- G_1(a, b) \circ c + G_1(b, a) \circ c - b \circ G_1(a, c) = 0,\notag \\
&&\label{G3}G_1(a, b) \circ c - G_1(a, c) \circ b + G_1(a \circ b, c) - G_1(a \circ c, b) = 0,\\
&&\label{G4}G_1(a, b) * c - G_1(a, c) * b + G_1(a * b, c) - G_1(a * c, b) = 0.
\vspb
\end{eqnarray}}

\begin{lem}
Let $(A,\cdot,\circ)$ be a {\transNovikovpoisson algebra}. Then \delete{both $G_1(a, b) = a \cdot b$ and $G_1(a, b) = x \cdot a \cdot b$ are the compatible global deformations of $(A, \circ)$, where $p$ is a fixed element in ${\bf k}$ or in $A$. Hence,} $(A, \circ_{x})$ is a Novikov algebra defined by
\begin{eqnarray*}
x \circ_{p} y := x \circ y + p \cdot x \cdot y\ \;\;\;\text{for all $x$, $y\in A$,}
\end{eqnarray*}
where $p \in {\bf k}$ or $p \in A$.
\end{lem}

\begin{proof}
It is straightforward.
\delete{The Eq.~\eqref{G1} holds since
\begin{eqnarray*}
&& a \cdot (b \circ c) - (a \circ b) \cdot c + (b \circ a) \cdot c - b \cdot (a \circ c) + a \circ (b c) - (ab) \circ c + (ba) \circ a - b \circ (ac) \\
&&= \frac{1}{2}(a \cdot b) \circ c + \frac{1}{2} b \circ (a \cdot c) - \frac{1}{2}(ac) \circ b - \frac{1}{2} a \circ (bc) + \frac{1}{2}(bc) \circ a\\
&&+ \frac{1}{2} b \circ (ac)- \frac{1}{2}(ab) \circ c - \frac{1}{2} a \circ (bc) - b \circ (ac) + a \circ (bc) \\
&&= 0.
\end{eqnarray*}
Similar reason apply to the proofs of Eqs.~\eqref{G2} -~\eqref{G4}.\vspb}
\end{proof}
\begin{rmk}
Note that $(A, \circ_{p})$ is a deformation of $(A,\circ)$.
\end{rmk}

\begin{thm}\label{global deformation thm}
Let $(A,\cdot,\circ)$ be a {\transNovikovpoisson} algebra. Then for any fixed elements $p$, $q\in A$ or ${\bf k}$, the triple $(A,\cdot_p,\circ_q)$ is a {\transNovikovpoisson algebra}, where
\begin{eqnarray*}
x\cdot_p y=p\cdot x\cdot y,\;\; x\circ_qy=x\circ y+q\cdot x\cdot y\;\;\;\text{for all $x,y\in A$}.
\end{eqnarray*}
\end{thm}
\begin{proof}
It can be checked directly.
\end{proof}

Next, we present a construction of {\transNovikovpoisson algebras} using Kantor products.

We first recall the definition of Kantor product of two binary operations. For the details, one can refer to \cite{Kantor1, Kantor2}.
\delete{The study of a specific class of algebras gives rise to the concept known as the Kantor product of multiplications. Originally introduced by Kantor in 1972, the family of conservative algebras~\cite{Kantor1} encompasses several fundamental classes of non-associative algebras. \par
Here we give only a brief overview of the Kantor product; see \cite{Kantor2}, for further details. }
Let \( V_n \) be a \( n \)-dimensional vector space, and denote by \( U(n) \) the space of all binary operations on \( V_n \). For a fixed element \( u \in V_n \) and any two binary operations \( A, B \in U(n) \), we define a new binary operation \( [A, B]_u \) on \( V_n \) by setting, for all \( x, y \in V_n \),
\begin{eqnarray*}
&&[A, B]_u(x, y) = A(u, B(x, y)) - B(A(u, x), y) - B(x, A(u, y)).
\end{eqnarray*}
$[A, B]_u$ is called the {\bf left Kantor product} of \( A \) and \( B \).\delete{ a corresponding right Kantor product can also be defined. Note that all products considered in this subsection are left Kantor products.}

\begin{pro}
Let $(A,\cdot, \circ)$ be a {\transNovikovpoisson} algebra and $u$ be a fixed element in $A$. Then $(A, [\cdot, \circ]_u)$ is a Novikov algebra and $(A,[\circ, \cdot]_u)$ is an associative commutative algebra.
\end{pro}

\begin{proof}
Firstly, for all $x$, $y\in A$, define
\begin{eqnarray*}
&&x * y:=[\circ,\cdot]_u(x,y),\;\;\; x \star y:=[\cdot, \circ]_u(x,y).
\end{eqnarray*}
Then we obtain
\begin{eqnarray*}
x * y &=&u \circ (x\cdot y) - (u \circ x)\cdot y - x\cdot (u \circ y) \\
&= &u \circ (x\cdot y) - 2x\cdot(u \circ y) = -(x\cdot u) \circ y = -(x\cdot y) \circ u,
\end{eqnarray*}
and
\begin{eqnarray*}
x \star y&=& u\cdot(x \circ y) - (u\cdot x) \circ y - x \circ (u\cdot y) = -u\cdot(x \circ y).
\end{eqnarray*}

For all $x,y,z\in A$, we have
\begin{eqnarray*}
(x * y) * z&=&(((x\cdot y)\circ u)\cdot z)\circ u=(((x\cdot y) \circ z)\cdot u)\circ u\\
&=&(((z\cdot y) \circ x)\cdot u)\circ u=(x\cdot ((y\cdot z) \circ u))\circ u=x * (y * z),
\end{eqnarray*}
and
\begin{eqnarray*}
x * y&=& -(x\cdot y) \circ u=y*x.
\end{eqnarray*}
Therefore, $(A, \ast)$ is a commutative associative algebra.

\delete{Secondly, define \vspa
\begin{eqnarray*}
&&x \diamond y:=[\cdot, \circ](x,y)= u(x \circ y) - (ux) \circ y - x \circ (uy) = -u(x \circ y).
\vspb
\end{eqnarray*}}
Since
\begin{eqnarray*}
&&(x \star y) \star z= - (u\cdot(x \circ y)) \star z = u\cdot((u\cdot(x \circ y)) \circ z) = u\cdot(((x \circ y) \cdot z) \circ u)=(x \star z) \star y,\\
&&y\star (x \star z) = y \star(-u\cdot(x \circ z)) = u\cdot(y \circ (u\cdot(x \circ z))),
\end{eqnarray*}
 Eq.~\eqref{Novikov1} holds. Moreover, we obtain
\begin{eqnarray*}
&&(x \star y) \star z-(y \star x) \star z-x\star (y \star z)+y\star (x \star z)\\
&&\quad=u\cdot\big((u\cdot(x \circ y)) \circ z - ((y \circ x)\cdot u) \circ z - x \circ (u\cdot(y \circ z)) + y \circ (u\cdot(x \circ z))\big)\\
&&\quad= u\cdot\big(\frac{1}{2}((u\cdot x) \circ y) \circ z + \frac{1}{2}(x \circ (u\cdot y)) \circ z - \frac{1}{2}((u\cdot y) \circ x) \circ z - \frac{1}{2}(y \circ (u\cdot x)) \circ z \\
&&\qquad- \frac{1}{2}x \circ ((u\cdot y) \circ z) - \frac{1}{2}x \circ ((y\cdot u) \circ z) - \frac{1}{2}y \circ ((u\cdot x) \circ z) + \frac{1}{2}y \circ ((x\cdot u) \circ z)\big) \\
&&\quad= 0.
\end{eqnarray*}
Therefore, $(A, \star)$ is a Novikov algebra.
\end{proof}
%The following corollaries provide methods to construct new {\transNovikovpoisson} from known ones.\vspb
\begin{thm}\label{kantorcor}
Let $(A, \cdot, \circ)$ be a {\transNovikovpoisson algebra} and $u$ be a fixed element in $A$. Then $(A, \cdot, [\cdot,\circ]_u)$ is also a {\transNovikovpoisson algebra}.
\end{thm}
\begin{proof}
Denote $[\cdot,\circ]_u$ by $\star$. By the definition of {\transNovikovpoisson algebras}, we obtain
\begin{eqnarray*}
 2z\cdot (x \star y) &=& 2z\cdot (-u\cdot(x \circ y)) = -2z\cdot u\cdot (x \circ y) \\
 &=& -u\cdot ((z\cdot x) \circ y) - u\cdot (x \circ (z\cdot y)) = (z\cdot x) \star y + x \star (z\cdot y),
 \end{eqnarray*}
 and
 \begin{eqnarray*}
 &&(x\cdot y)\star z = -u((x\cdot y) \circ z) = -u((x\cdot z) \circ y) = (x\cdot z) \star y.
\end{eqnarray*}
Therefore, $(A, \cdot, [\cdot,\circ]_u)$ is a transposed Novikov-Poisson algebra.
\end{proof}

%By Theorem~\ref{global deformation thm} and corollary~\ref{kantorcor}, we have the following corollary directly. \vspb
\begin{cor}
Let $(A,\cdot,\circ)$ be a {\transNovikovpoisson algebra}. For fixed elements $p,q,r\in {\bf k}$ or $A$, the triple $(A,\cdot_p,\circ_q)$ is also a {\transNovikovpoisson algebra}, where $x\cdot_py=x\cdot y\cdot p$, $x\circ_{q,r} y=q\cdot(x\circ y)+r\cdot x\cdot y$, for all $x, y\in A$.
\end{cor}
\begin{proof}
It follows directly by Theorems~\ref{global deformation thm} and \ref{kantorcor}.
\end{proof}

\section{$\frac{1}{2}$-derivations of Novikov algebras and {\transNovikovpoisson algebra} structures }

In this section, we present a relation between $\frac{1}{2}$-derivations of the associated Novikov algebras and transposed Novikov-Poisson algebras and investigate transposed Novikov-Poisson algebra structures on solvable Novikov algebras.
%{\transNovikovpoisson algebra} has characteristic $p\neq2$.
\vspace{-.2cm}

\subsection{$\frac{1}{2}$-derivations of Novikov algebras}
\label{1/2}
We start with the definition of $\frac{1}{2}$-derivations which is a particular case of $\delta$-derivations.
\vspace{-.1cm}
\begin{defi}\label{derivation}
Let $(A, \circ)$ be a Novikov algebra, $\varphi: A\rightarrow A$ be a linear map and $\delta\in {\bf k}$. Then $\varphi$ is a {\bf $\delta$-derivation} of $(A, \circ)$ if it satisfies
\begin{eqnarray}
\varphi(x \circ y) = \delta(\varphi(x) \circ y + x \circ \varphi(y))\;\;\;\;\text{for all}\: x, y \in A.
\end{eqnarray}
\end{defi}

\begin{rmk}
Let $\alpha\in {\bf k}$. Define $\varphi_\alpha: A\rightarrow A$ by
\begin{eqnarray}
\varphi_\alpha(x)=\alpha x\;\;\;\text{for all $x\in A$.}
\end{eqnarray}
Obviously, $\varphi_\alpha$ is a $\frac{1}{2}$-derivation of $(A,\circ)$. We call such  $\frac{1}{2}$-derivations of $(A,\circ)$ {\bf trivial}.
\delete{For an algebra $(A,\circ)$, the main examples of $\frac{1}{2}$-derivations of $(A,\circ)$ are the scalar multiplication maps by elements of the ground field. We call such $\frac{1}{2}$-derivations as trivial $\frac{1}{2}$-derivations. For a {\transNovikovpoisson algebra}, the Lemma above provides some natural examples of $\frac{1}{2}$-derivations.
\vspa}
\end{rmk}

Transposed Novikov-Poisson algebras are related with $\frac{1}{2}$-derivations of Novikov algebras.
\begin{lem}\label{1/2derivationintheT}
Let $(A, \cdot, \circ)$ be a {\transNovikovpoisson algebra} and $z$ be an arbitrary element in $A$. Then $L_\cdot(z)$ (or $R_\cdot(z)$) is a $\frac{1}{2}$-derivation of the Novikov algebra $(A, \circ)$.
\end{lem}
\begin{proof}
It follows immediately from Eq.~\eqref{transNP2}.
\end{proof}

By Lemma \ref{1/2derivationintheT}, \delete{With the relation between $\frac{1}{2}$-derivations of Lie algebras and transposed Poisson algebras established, it allows one to consequently construct certain transposed Poisson structures on corresponding Lie algebras. Similarly,} we can use $\frac{1}{2}$-derivations of Novikov algebras to construct {\transNovikovpoisson structures}. \delete{For a Novikov algebra of dimension one, its {\transNovikovpoisson structures} are completely determined. Hence, we exclude the one-dimensional cases from our consideration.
\vspa}

\begin{ex}
Let $A$ be a vector space over ${\bf k}$ with a basis $\mathcal{B} = \{ e_i \mid i \in \mathbb{Z} \}$. Define a binary operation $\circ:{A} \otimes {A} \to {A}$ such that $e_i \circ e_j = e_{i+j}$, for any $e_i, e_j \in \mathcal{B}$. Then $(A,\circ)$ is a Novikov algebra. Let $\varphi$ be a $\frac{1}{2}$-derivation of $(A,\circ)$. \delete{Since
\begin{eqnarray*}
 && 2\varphi(e_i) = 2\varphi(e_i\circ e_0) = \varphi(e_i)\circ e_0 + e_i\circ\varphi(e_0)= \varphi(e_i) + e_i\circ\varphi(e_0),
\end{eqnarray*}
we have $\varphi(e_i)= e_i\circ\varphi(e_0)$. Let $\cdot $ be a binary operation on $A$. Let $L_\cdot (e_j)(e_i)=e_i\circ\varphi(e_0)$

Let $(A, \cdot, \circ)$ be a {\transNovikovpoisson algebra}  with the Novikov algebra $(A, \circ)$ given above. By Proposition \ref{1/2derivationintheT}, we have $L_\cdot (e_i)=e_i\circ\varphi(e_0)$. }
Set $\varphi(e_i) = \sum_{j \in \mathbb{Z}} a_{i,j} e_j$ for each $e_i\in \mathcal{B}$. Since
\begin{eqnarray*}
 &&2\sum_{j \in \mathbb{Z}} a_{i,j} e_j = 2\varphi(e_i) = 2\varphi(e_i\circ e_0) = \varphi(e_i)\circ e_0 + e_i\circ\varphi(e_0)= \sum_{j \in \mathbb{Z}} a_{i,j} e_j + \sum_{j \in \mathbb{Z}} a_{0,j} e_{i+j},
\end{eqnarray*}
 we have $\varphi(e_i)= \sum_{j \in \mathbb{Z}} a_{0,j} e_{i+j}$ for each $i\in\mathbb{Z}$. Then there is a finite set $\{a_i\}_{i \in \mathbb{Z}}$ of elements from ${\bf k}$, such that $\varphi(e_i) = \sum_{j \in \mathbb{Z}} a_j e_{i+j}$. Then we consider the compatible commutative associative algebra structure $(A,\cdot)$ on $(A,\circ)$.
 Denote $L_\cdot(e_n)$ by $L_n$. By Lemma ~\ref{1/2derivationintheT}, there exist a finite set $\{a_{i,k}\}_{k \in \mathbb{Z}}$ of elements from ${\bf k}$ for each $i$, such that $L_j(e_i) =\sum_{k \in \mathbb{Z}} a_{j,k} e_{i+k}$. Note that
  $\sum_{k \in \mathbb{Z}} a_{j,k} e_{i+k} = L_j(e_i) = e_j \cdot e_i=e_i \cdot e_j = L_i(e_j) = \sum_{k \in \mathbb{Z}} a_{i,k} e_{k+j}$. Let $j=0$. Then we have $a_{i,k}=a_{0,k-i}$ for all $i,k\in \mathbb{Z}$. Thus we obtain $e_i \cdot e_j = \sum_{k \in \mathbb{Z}} a_{i,k} e_{k+j} = \sum_{k \in \mathbb{Z}} a_{0,k-i} e_{k+j}=  \sum_{k \in \mathbb{Z}} a_{0,k} e_{k+i+j}$.
For every finite set $\{a_l\}_{l\in \mathbb{Z}}$ of elements from ${\bf k}$, we define $e_i\cdot e_j= \sum_{k \in \mathbb{Z}} a_k e_{k+i+j}$. It is direct to check that triple $(A,\cdot,\circ)$ is a {\transNovikovpoisson algebra}.
\end{ex}

\begin{thm}\label{1/2-trans}
Let $(A,\circ)$ be a Novikov algebra with $\operatorname{dim}(A)\geq 2$. If there exist only trivial $\frac{1}{2}$-derivations on $(A,\circ)$, then every {\transNovikovpoisson} structure on $(A,\circ)$ is trivial.
\end{thm}

\begin{proof}
Let  $(A,\cdot,\circ)$ be a {\transNovikovpoisson algebra}. By Lemma \ref{1/2derivationintheT},  $R_\cdot(y)$ and $L_\cdot(x)$ are $\frac{1}{2}$-derivations of $(A,\circ)$ for all $x,y\in A$. By the assumption, there exist \(\kappa_x, \kappa_y \in {\bf k}\), such that \(L_\cdot(x)(z) = \kappa_x z\) and \(R_\cdot(y)(z) = \kappa_y z\)  for all \(z \in A\). Since $\operatorname{dim}(A)\geq 2$, we can take \(x\) and \(y\) as linearly independent elements. It follows that \(\kappa_x = \kappa_y = 0\), since $L_\cdot(x)(y)= \kappa_x y=x\cdot y= \kappa_y x=R_\cdot(y)(x)$. Therefore, we have $x\cdot z=L_\cdot(x)(z)=\kappa_xz=0$ for all $x,z\in A$.
\end{proof}

\begin{lem}\label{1/2id}
Let $(A,  \circ)$ be a Novikov algebra and $\varphi$ be a $\frac{1}{2}$-derivation of $(A, \circ)$. Then the following identities hold for all $x$, $y$, $z \in A$:
\begin{eqnarray}
&\label{1/2id1}(x \circ y) \circ \varphi(z) = (x \circ \varphi(y)) \circ z,\\
&\label{1/2id2}(x \circ y) \circ \varphi(z) - (y \circ x) \circ \varphi(z) = \varphi(x) \circ (y \circ z) - \varphi(y) \circ (x \circ z).
\end{eqnarray}

\end{lem}

\begin{proof}
For all $ x, y, z \in A$, we have
\begin{eqnarray*}
\varphi\bigl((x \circ y) \circ z\bigr)& =& \frac{1}{2} \varphi(x \circ y) \circ z + \frac{1}{2} (x \circ y) \circ \varphi(z)\\
&=& \frac{1}{4} (x \circ \varphi(y)) \circ z + \frac{1}{4} (\varphi(x) \circ y) \circ z + \frac{1}{2} (x \circ y) \circ \varphi(z) ,
\end{eqnarray*}
and
\begin{eqnarray*}
\varphi\bigl((x \circ z) \circ y\bigr)& =& \frac{1}{2} \varphi(x \circ z) \circ y + \frac{1}{2} (x \circ z) \circ \varphi(y)\\
&= &\frac{1}{4} (x \circ \varphi(z)) \circ y + \frac{1}{4} (\varphi(x) \circ z) \circ y + \frac{1}{2} (x \circ z) \circ \varphi(y).
\end{eqnarray*}
Taking the difference of the two identities above  gives Eq.~\eqref{1/2id1}.

Note that
\begin{eqnarray*}
\varphi\bigl((x \circ y) \circ z\bigr)-\varphi\bigl((y \circ x) \circ z\bigr)&=&\frac{1}{4} (x \circ \varphi(y)) \circ z + \frac{1}{4} (\varphi(x) \circ y) \circ z + \frac{1}{2} (x \circ y) \circ \varphi(z)\\
&&\quad-\frac{1}{4} (y \circ \varphi(x)) \circ z - \frac{1}{4} (\varphi(y) \circ x) \circ z - \frac{1}{2} (y \circ x) \circ \varphi(z),
\end{eqnarray*}
and
\begin{eqnarray*}
\varphi\bigl(x \circ (y \circ z)\bigr)- \varphi\bigl(y \circ (x \circ z)\bigr)&=&\frac{1}{4} x \circ (y \circ \varphi(z)) + \frac{1}{4} x \circ (\varphi(y) \circ z) + \frac{1}{2} \varphi(x) \circ (y \circ z)\\
&&\quad-\frac{1}{4} y \circ (x \circ \varphi(z)) - \frac{1}{4} y \circ (\varphi(x) \circ z) - \frac{1}{2} \varphi(y) \circ (x \circ z).
\end{eqnarray*}
By Eqs.~\eqref{Novikov1} and~\eqref{Novikov2}, we have
\begin{eqnarray*}
0&=&\varphi\bigl((x \circ y) \circ z\bigr)-\varphi\bigl((y \circ x) \circ z\bigr)-\varphi\bigl(x \circ (y \circ z)\bigr)+ \varphi\bigl(y \circ (x \circ z)\bigr)\\
&=&\frac{1}{4}\varphi(y)\circ(x\circ z)-\frac{1}{4}\varphi(x)\circ(y\circ z)+\frac{1}{4}(x\circ y)\circ\varphi(z) - \frac{1}{4}(y\circ x)\circ\varphi(z).
\end{eqnarray*}
This completes the proof.
\end{proof}
\delete{
More generally, this Lemma remains valid for an arbitrary $\delta$-derivation $\varphi$ of the Novikov algebra over a field with characteristic 0, the proof follows similar arguments.
\vspace{+,2cm}}

Finally, we show a relationship between transposed Novikov-Poisson algebras and Hom-Novikov algebras.
\begin{defi}~\cite{HomNovikov}
A {\bf Hom-Novikov algebra} is a triple $(A, \circ, \alpha)$ consisting of a vector space $A$, a binary operation $\circ : A \otimes A \longrightarrow A$ and an algebra homomorphism $\alpha : A \longrightarrow A$ satisfying
\begin{eqnarray}
&(x\circ y)\circ\alpha(z) - \alpha(x)\circ(y\circ z) = (y\circ x)\circ\alpha(z) - \alpha(y)\circ(x\circ z),\\
&(x\circ y)\circ\alpha(z) = (x\circ z)\circ\alpha(y)\;\;\;\text{for all $x$, $y$, $z\in A$.}
\end{eqnarray}
\end{defi}

\begin{pro}
Let $(A,\cdot,\circ)$ be a {\transNovikovpoisson algebra} and $p\in A$. If $L_\cdot(p)$ is an algebra homomorphism of $(A,\circ)$, then $(A,\circ, L_\cdot(p))$ is a Hom-Novikov algebra.
\end{pro}

\begin{proof}
By Lemma~\ref{1/2derivationintheT}, $L_\cdot(p)$ is a $\frac{1}{2}$-derivation of $(A,\circ)$. Then by Lemma ~\ref{1/2id}, for all $x$, $y$, $z\in A$, we obtain
\begin{eqnarray*}
&(x \circ y) \circ (L_\cdot(p)(z)\bigr)-(y \circ x) \circ (L_\cdot(p)(z)\bigr)=(L_\cdot(p)(x)) \circ (y \circ z)-(L_\cdot(p)(y)) \circ (x \circ z),&\\
&(x \circ y) \circ (L_\cdot(p)(z)\bigr) = (x \circ( L_\cdot(p)(y))) \circ z.
\end{eqnarray*}
Therefore, $(A,\circ,L_\cdot(p))$ is a Hom-Novikov algebra.
\end{proof}

\subsection{$\frac{1}{2}$-derivations of solvable Novikov algebras}
\delete{In ~\cite{1/2sol1}, the study of $\frac{1}{2}$-derivations on solvable Lie algebras was used to construct transposed Poisson algebra structures on solvable Lie algebras. Similarly,} In this subsection, we  consider $\frac{1}{2}$-derivations and {\transNovikovpoisson algebra structures} on solvable Novikov algebras. We first recall some necessary notions.
\begin{defi}
Let $(A,\circ)$ be a Novikov algebra. Define \( A^{(0)} = A \) and define \( A^{(n)} = A^{(n - 1)} \circ A^{(n - 1)} \) for every $n\in \mathbb{Z}_+$. Then \( (A,\circ) \) is called {\bf solvable} if there exists some $n\in \mathbb{Z}_+$ such that \( A^{(n)} = \{0\} \). Define \( A_L^1 = A \) and  \( A_L^n = A_L^{n - 1}\circ A \). Then \( (A,\circ) \) is called {\bf right nilpotent} if there exists some $n\in \mathbb{Z}_+$ such that \( A_L^n = \{0\} \). {\bf Left nilpotency} can be  defined similarly. Define \( A^1 = A \) and \( A^n = \sum_{1 \leq i \leq n - 1} A^i \circ A^{n - i} \), \( n \geq 2 \). Then \(( A,\circ )\) is called {\bf nilpotent} if \( A^n = \{0\} \) for some positive integer \( n \).
\end{defi}

\begin{rmk}\label{rmk-equiv}
By \cite[Theorem 3.3]{Zhang}, $(A,\circ)$ is right nilpotent if and only if $(A^2, \circ)$ is nilpotent if and only if $(A,\circ)$ is solvable.
\end{rmk}
\delete{The following Theorem proven in \cite{Zhang} reveals the relationships between the solvability of a (left) Novikov algebra and its left nilpotency, right nilpotency, and nilpotency.
\vspa
\begin{thm}\label{thm:zhang_key_result}
Let $(A,\circ)$ be a (left) Novikov algebra. Then the following conditions are equivalent:
\begin{enumerate}
    \item $A$ is right nilpotent.
    \item $A^2$ is nilpotent.
    \item $A$ is solvable.
\end{enumerate}
\vspa
\end{thm}}

Let $(A,\circ)$ be a Novikov algebra and \( I \subset A \). Recall that the {\bf annihilator}, {\bf left annihilator} and {\bf right annihilator} of \( I \) in \( A \) are $\operatorname{Ann}_A(I)=\{ a \in A \mid a\circ I =I\circ a= 0 \}$, $\operatorname{Ann}_L(I) = \{ a \in A \mid a\circ I = 0 \}$ and $\operatorname{Ann}_R(I) = \{ a \in A \mid I\circ a = 0 \}$ respectively.

\begin{lem}\label{Ann}\cite[Lemma 2]{Ann}, \cite[Lemma 2.1]{Zhang}
Let $(A,\circ)$ be a Novikov algebra. Then the following conclusion holds.
 \begin{enumerate}
 \item If \( I \) is a left ideal of \( (A,\circ) \), then \( \text{Ann}_L(I) \) is an ideal of \( (A, \circ) \).
  \item If \( I \) is an  ideal of \( (A,\circ) \), then \( \text{Ann}_A(I) \) is an ideal of \( (A,\circ) \).
\end{enumerate}
\end{lem}
\delete{\begin{proof}
\delete{Since the first part is already proofed in ~\cite{Ann}, we only prove the second part here.
By Eq.~\eqref{Novikov1} and ~\eqref{Novikov2}, we have\vspa
\begin{eqnarray*}
&&(a\circ x)\circ i=(a\circ i)\circ x=0,\quad i\circ (a\circ x)=(i\circ a)\circ x+a\circ(i\circ x)-(a\circ i)\circ x=0,\\
&&(x\circ a)\circ i=(x\circ i)\circ a=0,\quad i\circ (x\circ a)=(i\circ x)\circ a+x\circ(i\circ a)-(x\circ i)\circ a=0,\vspa
\end{eqnarray*}
for all $x\in \text{Ann}_A(I),\:a\in A,\:i\in I$. Hence, we obtain $a\circ x,\:x\circ a\in \text{Ann}_A(I)$.
\vspa}
\end{proof}}

\begin{lem}\label{decomposable}
Let $(A, \circ)$ be a decomposable Novikov algebra, i.e., $A$ is the direct sum of two non-zero ideals. Then $(A,\circ)$ has non-trivial $\frac{1}{2}$-derivations.
\end{lem}

\begin{proof}
By the hypothesis, we have $\operatorname{dim}(A)\geq2$. Assume that $A=A_1\oplus A_2$, where $A_1,A_2$ are non-zero ideals of $(A,\circ)$. For any $x_i=y_i+z_i\in A$, where $y_i\in A_1$, $z_i\in A_2$, define a linear map $\varphi:A \rightarrow A$ by $\varphi(x_i)=z_i$.
Then for any $x_i$, $x_j\in A$, we have
\begin{eqnarray*}
&&\varphi(x_i\circ x_j)=\varphi((y_i+z_i)\circ (y_j+z_j))=\varphi(y_i\circ y_j+z_i\circ z_j)=z_i\circ z_j,\\
&&\frac{1}{2}\varphi(x_i)\circ x_j+\frac{1}{2}x_i\circ \varphi(x_j)=\frac{1}{2}z_i\circ x_j+\frac{1}{2}x_i\circ z_j=\frac{1}{2}z_i\circ z_j+\frac{1}{2}z_i\circ z_j.
\end{eqnarray*}
Therefore, the linear map $\varphi$ defined above is a non-trivial $\frac{1}{2}$-derivation of $(A,\circ)$. Then this conclusion holds.
\end{proof}

\begin{lem}\label{1/2-solvable1}
Let $(A,\circ)$ be a finite-dimensional Novikov algebra with $\operatorname{Ann}_A(A)\neq 0$ and $A\circ A\neq A$. Then $A$ has non-trivial $\frac{1}{2}$-derivations.
\end{lem}

\begin{proof}
When $A\circ A=0$,  all linear maps $\varphi:A\rightarrow A$ are $\frac{1}{2}$-derivations. Then this conclusion holds.

Next, we only need to consider the case when $A\circ A\neq0$. Since $A\circ A\neq A$, we have $\operatorname{dim}(A)\geq 2$.

If $(A\circ A) \cap \operatorname{Ann}_{A}(A) = 0$, then there exists a subspace $V$ of $A$ such that $A=V\oplus (A\circ A)\oplus \operatorname{Ann}_A(A)$  as vector spaces, where $V$ can be the zero vector space. Since $V\oplus (A\circ A)$ and $\operatorname{Ann}_A(A)$ are non-zero ideals of $(A,\circ)$, then this conclusion follows by Lemma~\ref{decomposable}.

If $(A\circ A) \cap \operatorname{Ann}_{A}(A) \neq0$, choose a non-zero element $x_1\in (A\circ A) \cap \operatorname{Ann}_{A}(A)$. Moreover, there exists a non-zero subspace $W$ of $A$, such that $A=(A\circ A)\oplus W$ as vector spaces. \delete{Then we can extend $x_1$ into a linear basis $X=\{x_1,...,x_n\}$ of $A\circ A$ and assume that $Y=\{y_1,...,y_m\}$ is a linear basis of $W$. Next,} Define a linear map $\varphi:A \rightarrow A$ by $\varphi(x)=0,\:\varphi(y)=x_1$, for all $x\in A\circ A$ and $y\in W$. Then for all $z_1$, $z_2\in A$, we have
\begin{eqnarray*}
&&\varphi(z_1\circ z_2)=0=\frac{1}{2}\varphi(z_1)\circ z_2+\frac{1}{2}z_1\circ \varphi(z_2).
\end{eqnarray*}
Therefore, $\varphi$ defined above is a non-trivial $\frac{1}{2}$-derivation of $(A,\circ)$.
\end{proof}

Based on the construction of $\frac{1}{2}$-derivations in Lemma \ref{1/2-solvable1}, we obtain the following theorem.
\begin{thm}\label{Annnot0}
Let $(A,\circ)$ be a finite-dimensional Novikov algebra. If $A\circ A\neq 0$, $A\circ A\neq A$ and $\operatorname{Ann}_A(A)\neq 0$, then there are non-trivial {\transNovikovpoisson structures} on $(A,\circ)$.
\end{thm}

\begin{proof}
We prove it with the same notations as in the proof of Lemma~\ref{1/2-solvable1}.\delete{ When $A\circ A=0$, the conclusion holds clearly. Next, we consider $A\circ A\neq0$, then}
 By the assumption, we have $\operatorname{dim}(A)\geq 2$.

If $(A\circ A) \cap \operatorname{Ann}_{A}(A) = 0$,  let $Y = \{ y_1, \ldots, y_m \}$ be a basis of $V\oplus (A\circ A)$ and $X = \{x_1, \ldots, x_n \}$ be a basis of $\operatorname{Ann}_A(A)$. Define a binary operation $\cdot:A\otimes A \rightarrow A$ by $x_i\cdot x_j=x_i+x_j$, and $z_1\cdot z_2=0$, if $\{z_1,z_2\}\subset X\cup Y$ and $\{z_1,z_2\}\nsubseteq X$. Obviously, the binary operation $\cdot$ is commutative and associative. Moreover, for all $x,y,z\in A$, we have
\begin{eqnarray*}
&2z\cdot(x\circ y)\in 2z\cdot(A\circ A)=0,\:(z\cdot x)\circ y\in\operatorname{Ann_A}(A)\circ A=0,\\
&x\circ (z\cdot y)\in A\circ \operatorname{Ann_A}(A)=0,\:(x\cdot y)\circ z\:,(x\cdot z)\circ y\in \operatorname{Ann_A}(A)\circ A=0.
\end{eqnarray*}
Thus, Eqs.~\eqref{transNP1} and ~\eqref{transNP2} hold for the triple $(A,\cdot,\circ)$. Then $(A,\cdot,\circ)$ is a non-trivial transposed Novikov-Poisson algebra.

If $(A\circ A) \cap \operatorname{Ann}_{A}(A) \neq 0$, \delete{ then exists a non-zero element $y_1\in (A\circ A) \cap \operatorname{Ann}_{A}(A) $. So we may assume that $A=W\oplus (A\circ A)$ as vector spaces. } let $X = \{ x_1, \ldots, x_n \}$ be a basis of $A\circ A$ and $Y=\{y_1,\ldots,y_m\}$ be a linear basis of $W$. Define a binary operation $\cdot:A\otimes A\rightarrow A$ by $y_i\cdot y_j=x_1$, and $z_1\cdot z_2=0$ if $\{z_1,z_2\}\subset X\cup Y$ and $\{z_1,z_2\}\nsubseteq X$. Obviously, the binary operation $\cdot$ is commutative and associative. Then for all $x,y,z\in A$, similar to the previous case, by analogous computations, we obtain
\begin{eqnarray*}
&2z\cdot(x\circ y)=0=(z\cdot x)\circ y+x\circ (z\cdot y),\\
&(x\cdot y)\circ z=0=(x\cdot z)\circ y.
\end{eqnarray*}
Therefore,  $(A,\cdot,\circ )$ is a {\transNovikovpoisson} algebra.
\end{proof}

\delete{For a solvable Novikov algebra $(A,\circ)$, we have $A\circ A\neq A$, if not $A^{(n)}=A$ would hold for all $n\in \mathbb{N}$. By Theorem~\ref{Annnot0}, we obtain the following corollary directly.\vspa}
\begin{cor}
Let $(A,\circ)$ be a finite dimensional solvable Novikov algebra with $A\circ A\neq 0$ and $\operatorname{Ann}_{A}(A) \neq 0$. Then there is a non-trivial {\transNovikovpoisson algebra} structure on  $(A,\circ)$.
\end{cor}
\begin{proof}
Since $(A,\circ)$ is solvable, we have $A\circ A\neq A$. Then it follows directly by Theorem~\ref{Annnot0}.
\end{proof}

\begin{ex}
Let $A$ be a 3-dimensional vector space over $\mathbb{C}$ with a basis $\{e_1, e_2,e_3\}$.
By \cite{erwei}, $(A, \circ)$ is a transitive (right nilpotent) Novikov algebra with the non-zero products given by
\begin{eqnarray*}
&&e_3 \circ e_2 = e_2.
\end{eqnarray*}
By Remark \ref{rmk-equiv}, $(A,\circ)$ is solvable. By a straightforward computation, we have $\operatorname{Ann}_R(A)=\mathbb{C}e_1\oplus\mathbb{C}e_3$ and $\operatorname{Ann}_L(A)=\mathbb{C}e_1\oplus\mathbb{C}e_2$.  Thus, $\operatorname{Ann}_A(A)=\mathbb{C}e_1$ and $(A\circ A)\cap \operatorname{Ann}_A(A)=\{0\}$, since $(A\circ A)=\mathbb{C}e_2$. \delete{With the notations as in the proof of Theorem~\ref{Annnot0}, let $X=\{e_1\}$ and $Y=\{e_2,e_3\}$.} By applying Theorem~\ref{Annnot0}, define a commutative and associative algebra $(A,\cdot)$ whose non-zero products are given by $e_1\cdot e_1=2e_1$. Clearly, the triple $(A,\cdot,\circ)$ forms a {\transNovikovpoisson algebra}.
\end{ex}

To further illustrate Theorem \ref{Annnot0}, we provide an example of transposed Novikov-Poisson algebras with Novikov algebras which are not solvable but satisfy $A\circ A\neq A$.

\delete{To further illustrate Theorem~\ref{thm:zhang_key_result}, we will provide an example of a Novikov algebra that is not solvable but satisfies $A\circ A\neq A$.\vspa}
\begin{ex}
Let $A$ be a 3-dimensional vector space over $\mathbb{C}$ with a basis $\{e_1, e_2,e_3\}$. By \cite{erwei}, $(A, \circ)$ is not a transitive (right nilpotent) algebra with non-zero products given by
\begin{eqnarray*}
&&e_1 \circ e_1 = e_2,\quad e_3\circ e_3=e_3.
\end{eqnarray*}
By Remark \ref{rmk-equiv}, $(A,\circ)$ is not solvable. It is easy to see that $\operatorname{Ann}_R(A)=\mathbb{C}e_2$ and $\operatorname{Ann}_L(A)=\mathbb{C}e_2$. Therefore, we have $\operatorname{Ann}_A(A)=\mathbb{C}e_2$, and $(A\circ A)\cap \operatorname{Ann}_A(A)=\mathbb{C}e_2$, since $A\circ A=\mathbb{C}e_2\oplus \mathbb{C}e_3$.\delete{ With the notations as in the proof of Theorem~\ref{Annnot0}.} Let $X=\{x_1=e_2,e_3\}$ and $Y=\{e_1\}$. By applying Theorem~\ref{Annnot0}, define a commutative associative algebra $(A,\cdot)$, whose non-zero products are given by $e_1\cdot e_1=e_2$. Then the triple $(A,\cdot,\circ)$ is a {\transNovikovpoisson algebra}.
\end{ex}

Next, we will study $\frac{1}{2}$-derivations and  {\transNovikovpoisson algebra structures} on solvable Novikov algebras with $\operatorname{Ann}_A(A)=0$. However, some three-dimensional examples illustrate that for a solvable Novikov algebra with $\operatorname{Ann}_A(A)=0$, a non-trivial {\transNovikovpoisson algebra structure} may not necessarily exist.
\vspa
\begin{ex}
Let $A$ be a 3-dimensional vector space over $\mathbb{C}$ with a basis $\{e_1, e_2,e_3\}$. By \cite{erwei}, $(A, \circ)$ is a transitive (right nilpotent) algebra with the non-zero products given by
\begin{eqnarray*}
&&e_2 \circ e_2 = e_1,\quad e_3\circ e_1=e_1,\quad e_3\circ e_2=\frac{1}{2}e_2.
\end{eqnarray*}
By Remark \ref{rmk-equiv}, $(A,\circ)$ is solvable and $\operatorname{Ann}_A(A)=0$. By a straightforward computation, the {\transNovikovpoisson algebra} $(A,\cdot,\circ)$ is trivial.
\end{ex}

Finally, we provide a construction of $\frac{1}{2}$-derivations and the corresponding {\transNovikovpoisson} algebra structures on Novikov algebras with $\operatorname{Ann}_A(A) = 0$.

\begin{lem}\label{spec}
Let $(A,\circ)$ be a solvable Novikov algebra with $\operatorname{Ann}_A(A)=0$. If $\operatorname{Ann}_A(A\circ A)\setminus \operatorname{Ann}_L(A)\neq 0$, then $A$ has non-trivial $\frac{1}{2}$-derivations.
\end{lem}

\begin{proof}
First, we have $\operatorname{dim}(A)\geq2$. Choose an element $w\in \operatorname{Ann}_A(A\circ A)\setminus \operatorname{Ann}_L(A)$, we define a linear map $\varphi:A\rightarrow A$ by $\varphi(x)=w\circ x$. If there exists $k_w\in {\bf k}$ such that $\varphi(x)=w\circ x=k_wx$ for all $x\in A$, then we have $\varphi(x\circ y)=w\circ (x\circ y)=k_w(x\circ y)=0$ for all $x,y\in A$. Since $A\circ A\neq 0 $, we have $k_w=0=w\circ x$, which contradicts with $w\notin \operatorname{Ann}_L(A)$. Therefore,  $\varphi$ defined above is non-trivial.

Finally, we check that $\varphi$ is a $\frac{1}{2}$-derivation. Since
\begin{eqnarray*}
2\varphi(x\circ y)&=&2w\circ (x\circ y)=0,
\end{eqnarray*}
and
\begin{eqnarray*}
\varphi(x)\circ y+x\circ \varphi(y)&=&(w\circ x)\circ y+x\circ (w\circ y)\\
&=&(x\circ w)\circ y+w\circ (x\circ y)-x\circ(w\circ y)+x\circ (w\circ y)\\
&=&(x\circ y)\circ w+w\circ (x\circ y)\\
&=&0,
\end{eqnarray*} $\varphi$ is a non-trivial $\frac{1}{2}$-derivation.
\end{proof}

\begin{thm}
Let $(A,\circ)$ be a solvable Novikov algebra with $\operatorname{Ann}_A(A)=0$. If $\operatorname{Ann}_A(A\circ A)\setminus \operatorname{Ann}_L(A)\neq 0$, then there are non-trivial {\transNovikovpoisson algebra structures} on $(A,\circ)$.
\end{thm}
\begin{proof}
Following the proof of Lemma \ref{spec}, we define a binary operation
$\cdot:A\otimes A\rightarrow A$ by $x\cdot y=(w\circ x)\circ y$ for all $x$, $y\in A$, where $w\in \operatorname{Ann}_A(A\circ A)\setminus \operatorname{Ann}_L(A)$. By Eqs.~\eqref{Novikov1}, $(A,\cdot)$ is commutative. Then we have
\begin{eqnarray*}
&&x\cdot y=(w\circ x)\circ y=(x\circ w)\circ y+w\circ (x\circ y)-x\circ (w\circ y)=-x\circ (w\circ y),
\vspb
\end{eqnarray*}
for all $x,y\in A$. If $x\cdot y=0$ for all $x,y\in A$, we have $w\circ A\subset \operatorname{Ann}_L(A)\cap \operatorname{Ann}_R(A)$.  This directly implies that $w\circ A=0$ due to $\operatorname{Ann}_A(A)=0$, which contradicts with $w\notin \operatorname{Ann}_L(A)$. Hence, $(A,\cdot)$ defined above is not trivial.
Moreover, $(A,\cdot)$ defined above is associative, since
\begin{eqnarray*}
%&&x\cdot y=(w\circ x)\circ y=(w\circ y)\circ x=y\cdot x,\\
&(x\cdot y)\cdot z=(w\circ ((w\circ x)\circ y))\circ z=0,\\
&x\cdot (y\cdot z)=(w\circ x)\circ ((w\circ y)\circ z)=(w\circ ((w\circ y)\circ z))\circ x=0,
\end{eqnarray*}
for all $x,y,z\in A$. \delete{In the second equation above, $(w\circ ((w\circ x)\circ y))\circ z=0$ holds because $(w\circ x)\circ y \in A\circ A$. The third equation holds similarly.\par}

Finally, we need to check that Eqs.~\eqref{transNP1} and ~\eqref{transNP2} hold for $(A,\cdot,\circ)$. Since $\text{Ann}_A(A\circ A)$ is an ideal of $(A, \circ)$ by Lemma~\ref{Ann}, we have $w\circ A\subset \operatorname{Ann}_A(A\circ A)$. Therefore, we have $(x\circ y)\circ (w\circ z)=0$ for all $x$, $y$, $z\in A$. Then for all $x,y,z\in A$, we obtain
\begin{eqnarray*}
&&(x\cdot y)\circ z=\big((w\circ x)\circ y\big)\circ z=\big((w\circ x)\circ z\big)\circ y=(x\cdot z)\circ y,
\end{eqnarray*}
\begin{eqnarray*}
&&2z\cdot(x\circ y)=2(w\circ z)\circ (x\circ y)=2(w\circ  (x\circ y))\circ z=0,
\end{eqnarray*}
and \begin{eqnarray*}
(z\cdot x)\circ y+x\circ (z\cdot y)&=&((w\circ z)\circ x)\circ y+x\circ ((w\circ z)\circ y)\\
&&=(x\circ (w\circ z))\circ y+(w\circ z)\circ (x\circ y)-x\circ ((w\circ z)\circ y)+x\circ ((w\circ z)\circ y)\\
&&=(x\circ y)\circ (w\circ z)+(w\circ (x\circ y))\circ z\\
&&=0.
\end{eqnarray*}
\delete{In the last equation above, since $\text{Ann}_A(A\circ A)$ is an ideal of $(A, \circ)$ by Lemma~\ref{Ann}, we have $w\circ A\subset \operatorname{Ann}_A(A\circ A)$, i.e., $(x\circ y)\circ (w\circ z)=0$.} Then the proof is completed.
\end{proof}

\section{Simple {\transNovikovpoisson algebras}}
In this section, we investigate simple transposed Novikov-Poisson algebras. We show that if $(A, \cdot,\circ)$ is a non-trivial simple transposed Novikov-Poisson algebra, then $(A, \circ)$ is simple. Based on this result, transposed Novikov-Poisson algebras on some simple Novikov algebras are characterized.
\subsection{On simple {\transNovikovpoisson} algebras}
%We first give some background of simple Novikov.
Let $(A, \ast)$ be a Novikov algebra or a commutative associative algebra. $(A,\ast)$ is called {\bf simple} if $A \ast A \neq \{0\}$ and the only ideals of $A$ are {\bf trivial ideals}: $\{0\}$ and $A$.
\begin{defi}
\delete{The algebra $A$ is called {\bf simple} if $A \circ A \neq \{0\}$ and the only ideals of $A$ are the {\bf trivial ideals}: $\{0\}$, $A$.} An {\bf ideal} of a {\transNovikovpoisson algebra} $(A,\cdot,\circ)$ is a subspace \( I \) such that $I$ is an ideal of both $(A, \circ)$ and $(A,\cdot)$.
\delete{\( I\cdot A \subset I \) , \( I\circ A \subset I \) and $A\circ I\subset I$.} $0$ and $A$ are called {\bf trivial ideals} of $(A,\cdot,\circ)$.
A nontrivial {\transNovikovpoisson algebra} $(A,\cdot,\circ)$ is called {\bf simple} if $(A, \cdot, \circ)$ has no non-trivial ideals.
\end{defi}

\begin{lem}\label{simple-prefect}
Let $(A,\cdot,\circ)$ be a non-trivial {\transNovikovpoisson algebra}\delete{ with $A\circ A \neq0$}. If $A\circ A\neq A$, then both $(A,\cdot)$ and $(A,\cdot,\circ)$ are not simple.
\end{lem}

\begin{proof}
By Eq.~\eqref{transNP2}, we have
\begin{eqnarray*}
&&(A\circ A)\cdot A\subset(A\cdot A)\circ A+A\circ (A\cdot A)\subset A\circ A.
\end{eqnarray*}
Therefore, $A\circ A$ is a non-trivial ideal of $(A,\cdot)$. Note that $A\circ A$ is an ideal of $(A,\circ)$. Therefore, $A\circ A$ is a non-trivial ideal of $(A,\cdot,\circ)$.
\end{proof}

\begin{pro}\label{simpleT-prefect}
Let $(A,\cdot,\circ)$ be a non-trivial simple {\transNovikovpoisson algebra}. Then we have $A\circ A=A$.
\end{pro}
\begin{proof}
It follows directly from Lemma~\ref{simple-prefect}.
\end{proof}

\delete{Recall that a {\bf transposed quasi-ideal} of a transposed Poisson algebra $(A,\cdot,[,\:])$ introduced in ~\cite{simple-simple}, which is a non-trivial subspace \( I \) of \( A \) such that \( [A\cdot I,A] \subset I \) and \( [A, I] \subset I \). It is the transposed version of the notion of a quasi-ideal of a Poisson algebra. Similarly, we can define the transposed quasi-ideal of a {\transNovikovpoisson algebra}.\vspa}
Motivated by the study of simple transposed Poisson algebras in ~\cite{simple-simple}, we introduce the definition of transposed quasi-ideals of transposed Novikov-Poisson algebras.
\begin{defi}
A {\bf transposed quasi-ideal} of a {\transNovikovpoisson algebra} $(A,\cdot,\circ)$ is a non-trivial subspace \( I \) of \( A \) such that $I$ is an ideal of $(A,\circ)$ and \( (A\cdot I)\circ A \subset I \), $A\circ (A\cdot I)\subset I$.
\end{defi}
\begin{lem} \label{lem-1}
Let  $(A,\cdot,\circ)$ be a non-trivial simple transposed Novikov-Poisson algebra. Then a transposed quasi-ideal $I$ of $(A,\cdot,\circ)$ is an ideal of $(A,\cdot,\circ)$.
\end{lem}
\begin{proof}
Note that by the definition of transposed quasi-ideals, $I$ is an ideal of $(A,\circ)$. Moreover, by Theorem \ref{simpleT-prefect} and Eq.~\eqref{transNP2}, we have $A\cdot I=(A\circ A)\cdot I\subset(A\cdot I)\circ I+A\circ (A\cdot I)\subset I$. Therefore,  $I$ is also an ideal of $(A,\cdot)$. Then the conclusion holds.
\end{proof}
\begin{lem}\label{noquasiideal}
A non-trivial simple {\transNovikovpoisson algebra} $(A,\cdot,\circ)$ contains no transposed quasi-ideals.
\end{lem}

\begin{proof}
It follows directly from Lemma \ref{lem-1}.
\end{proof}
\delete{Suppose \( I \) is a transposed quasi-ideal of {\transNovikovpoisson algebra} \((A,\cdot,\circ)\). Let \( I' \) be a maximal subspace such that \( A \circ I' \subset I \) and \( I' \circ A \subset I \).
Firstly, we have
\begin{eqnarray*}
(A \cdot I) \circ A \subset I,\quad A \circ (A \cdot I) \subset I,\quad A \circ I \subset I,\quad I \circ A \subset I  \vspa
\end{eqnarray*}
by the definition of transposed quasi-ideal.
Then by the maximality of \( I' \), we have $0 \neq I \subset I' $ and \( A \cdot I \subset I' \). Next, we have $A\circ A\neq 0$ since $(A\cdot,\circ)$ is simple. And we obtain $A=A\circ A$ by Theorem \ref{simpleT-prefect}. By Eq.~\eqref{Tid1}, we have\vspa
\begin{eqnarray*}
&&A\cdot I'=(A\circ A)\cdot I'=(A\circ I')\cdot A\subset I\cdot A\subset I',\\
&&A\circ I'\subset I\subset I',\quad I'\circ A\subset I\subset I'.
\vspa
\end{eqnarray*}
And we have $I'\neq A$, if not $A=A\circ A=A\circ I'\subset I$, which contradicts the assumption. Thus,  \( I' \) is a non-trivial ideal of {\transNovikovpoisson algebra} \((A,\cdot,\circ)\), which contradicts the simplicity of \((A,\cdot,\circ)\).
\vspa}
\delete{\end{proof}
In fact, for a simple {\transNovikovpoisson algebra} $(A,\cdot,\circ)$, every transposed quasi-ideal $I$ is a non-trivial ideal of $(A,\cdot,\circ)$. To prove this conclusion, it is enough to show that $I$ is an ideal of $(A,\cdot)$. By Theorem \ref{simpleT-prefect} and Eq.~\eqref{transNP2}, we have $A\cdot I=(A\circ A)\cdot I\subset(A\cdot I)\circ I+A\circ (A\cdot I)\subset I$.
\vspd}

\begin{lem}\label{ideal-ideal}
Let $(A,\cdot,\circ)$ be a {\transNovikovpoisson algebra} with $A\circ A=A$. Then every non-trivial ideal of $(A,\circ)$ is a transposed quasi-ideal of $(A,\cdot,\circ)$.
\end{lem}

\begin{proof}
Suppose that $I$ is a non-trivial ideal of $(A,\circ)$. Since $A\circ I\subset I$ and $I\circ A\subset I$, by Eqs.~\eqref{Novikov1} and~\eqref{transNP1}, we have
\begin{eqnarray*}
&&(A\cdot I)\circ A\subset (A\cdot A)\circ I\subset A\circ I\subset I,
\end{eqnarray*}
and
\begin{eqnarray*}
A\circ (A\cdot I)&=&(A\circ A)\circ (A\cdot I)=(A\circ (A\cdot I))\circ  A\subset ((A\cdot I)\circ A)\circ A\\
&&\quad+A\circ ((A\cdot I)\circ A)-(A\cdot I)\circ (A\circ A)\subset I\circ A+A\circ I-(A\cdot I)\circ A\subset I.
\end{eqnarray*}
Therefore, $I$ is a transposed quasi ideal of $(A,\cdot,\circ)$.
\end{proof}

\begin{thm}\label{simplesimple}
Let $(A,\cdot,\circ)$ be a non-trivial simple {\transNovikovpoisson algebra}. Then the Novikov algebra $(A,\circ)$ is simple.
\end{thm}
\begin{proof}
By Proposition ~\ref{simpleT-prefect}, we have $A\circ A=A$.  By Lemma~\ref{ideal-ideal}, if $I$ is a non-trivial ideal of $(A,\circ)$, then $I$ is a transposed quasi-ideal of $(A,\cdot,\circ)$, which contradicts with Lemma~\ref{noquasiideal}.
\end{proof}

\begin{cor}
Let $(A,\cdot,\circ)$ be a {\transNovikovpoisson algebra}. Then the associated Novikov algebra $(A,\circ)$ is a direct sum of simple ideals if and only if the {\transNovikovpoisson} algebra is a direct sum of simple ideals.
\vspa
\end{cor}

\begin{proof}
Suppose that $A=\oplus_{i\in S} I_i$, where each $I_i$ is a simple ideal of $(A,\circ)$. Thus, we have $I_i\circ I_i=I_i$. By Eq.~\eqref{transNP2}, we obtain
\begin{eqnarray*}
A\cdot I_i=A\cdot(I_i\circ I_i)\subset (A\cdot I_i)\circ I_i+I_i\circ(A\cdot I_i)\subset A\circ I_i+I_i\circ A\subset I_i.
\end{eqnarray*}
Thus, each $I_i$ is an ideal of  $(A,\cdot,\circ)$. Finally, each $I_i$ as the ideal of $(A,\cdot,\circ)$ is simple, since $I_i$ is simple as the ideal of $(A,\circ)$. Conversely, this follows directly from Theorem~\ref{simplesimple}.
\end{proof}

It was shown in \cite{youxian0} that a finite-dimensional simple Novikov algebra over an algebraically closed field with characteristic 0 is one-dimensional. Note that a finite-dimensional simple commutative associative algebra over an algebraically closed field with characteristic 0 is also one-dimensional. Therefore, by Theorem~\ref{simplesimple}, we have the following corollary.
\begin{cor}\label{Cor-dim-1}
Let ${\bf k}$ be an algebraically closed field with characteristic 0. Then any simple finite dimensional {\transNovikovpoisson algebras} over ${\bf k}$ is one-dimensional.
\end{cor}

\subsection{Transposed Novikov-Poisson algebras with simple Novikov algebras}
By the discussion above, for classifying non-trivial simple transposed Novikov-Poisson algebras, we only need to consider the compatible commutative associative algebra structures on simple Novikov algebras.

%This subsection investigates the possibility of defining {\transNovikovpoisson} algebras on simple Novikov algebras.\par
If ${\bf k}$ is an algebraically closed field of characteristic zero, by Corollary \ref{Cor-dim-1}, finite-dimensional non-trivial simple transposed Novikov-Poisson algebras up to isomorphism are of the form that given in Example \ref{ex-dim-1}.

Next, let ${\bf k}$ be  an algebraically closed field  with characteristic $p > 2$ and $(A,\circ)$ be a finite-dimensional simple Novikov algebra with $\text{dim}(A)\geq 2$. By ~\cite{simplesushuyouxian},  $A$ has a basis $\{y_{-1}, y_0, \ldots, y_{p^n - 2}\}$ for some positive integer $n$ such that
\begin{eqnarray}\label{eq-simple-Nov}
&& y_i \circ y_j = \dbinom{i + j + 1}{j} y_{i+j} + \delta_{i, -1} \delta_{j, -1} a y_{p^n - 2} + \delta_{i, -1} \delta_{j, 0} b y_{p^n - 2}
\end{eqnarray}
for $i, j = -1, 0, \ldots, p^n - 2$, where $a, b \in \mathbf{k}$ are constants.

\begin{thm}\label{finitesimplethm}
There are no non-trivial $\frac{1}{2}$-derivations of finite-dimensional simple Novikov algebras with dimensions $\geq 1$ over an algebraically closed field  ${\bf k}$ with characteristic $p > 2$.
\end{thm}

\begin{proof}
Let $(A, \circ)$ be the simple Novikov algebra defined by Eq. (\ref{eq-simple-Nov}) and $\varphi$ be a $\frac{1}{2}$-derivation of $(A,\circ)$. Set $2\varphi(y_{i} ) =\sum_{i =-1}^{p^n - 2} \alpha_{i, j} y_j$, where $\alpha_{i,j}\in {\bf k}$.
Let $i=p^n-2,j=0$. Then we have
\begin{eqnarray*}
&&0= 2\varphi(y_{p^n-2} \circ y_{0})- \varphi(y_{p^n-2}) \circ y_{0} -y_{p^n-2} \circ \varphi(y_{0})\\
&&=2 \left( \sum_{k = -1}^{p^n - 2} \alpha_{p^n-2, k} y_k \right)- \left( \sum_{k = -1}^{p^n - 2} \alpha_{p^n-2, k} y_k \right) \circ y_{0} -y_{p^n-2}\circ \left( \sum_{k = -1}^{p^n - 2} \alpha_{0, k} y_k \right)\\
&&=2 \left( \sum_{k = -1}^{p^n - 2} \alpha_{p^n-2, k} y_k \right)- \sum_{k = -1}^{p^n - 2} \alpha_{p^n-2, k} y_k-\alpha_{p^n-2, -1}by_{p^n-2}-\alpha_{0,0}y_{p^n-2}.
\end{eqnarray*}
Thus, we have $\varphi(y_{p^n-2})=\alpha_{p^n-2,p^n-2}y_{p^n-2}=\alpha_{0,0}y_{p^n-2}$. Similarly, when $i=j=0$, we have\vspa
\begin{eqnarray*}
&&0=2\varphi(y_{0} \circ y_{0})-\varphi(y_{0}) \circ y_{0} - y_{0} \circ \varphi(y_{0})\\
&&=2 \left( \sum_{k = -1}^{p^n - 2} \alpha_{0, k} y_k \right)-\sum_{k = 0}^{p^n - 2} \alpha_{0, k}(k+1) y_k- \alpha_{0,-1}by_{p^n-2}-\left( \sum_{k = -1}^{p^n - 2} \alpha_{0, k} y_k \right).
\end{eqnarray*}
Comparing the coefficient of $y_{-1}$, we have $\alpha_{0,-1}=0$. Let $i=-1,j=-1$. Then we have\vspa
\begin{eqnarray*}
&&0=2\varphi(y_{-1} \circ y_{-1})-\varphi(y_{-1}) \circ y_{-1} -y_{-1} \circ \varphi(y_{-1})\\
&&=2a\alpha_{p^n-2,p^n-2}y_{p^n-2}-\sum_{k = -1}^{p^n - 2} \alpha_{-1, k} y_k\circ y_{-1}-y_{-1}\circ \sum_{k = -1}^{p^n - 2} \alpha_{-1, k} y_k\\
&&=2a\alpha_{p^n-2,p^n-2}y_{p^n-2}-2a\alpha_{-1,-1}y_{p^n-2}-\alpha_{-1,0}by_{p^n-2}- \sum_{k = 0}^{p^n - 2} \alpha_{-1, k} y_{k-1}.
\end{eqnarray*}
Thus we obtain $\alpha_{-1,k}=0$ when $k\neq -1$. Let $i=-1,j=0$. Then we have\vspa
\begin{eqnarray*}
&&0=2\varphi(y_{-1}\circ y_0)-y_{-1}\circ \varphi(y_0)-\varphi(y_{-1})\circ y_0\\
&&=2b\alpha_{0,0}y_{p^n-2}+2\alpha_{-1,-1}y_{-1}-\alpha_{0,0}by_{p^n-2}-\sum_{k = 0}^{p^n - 2}\alpha_{0,k}y_{k-1}-\alpha_{-1,-1}y_{-1}-\alpha_{-1,-1}by_{p^n-2}.
\end{eqnarray*}
Comparing the coefficients of $y_k,\:k=-1,0,...,p^n-1$, we have $\varphi(y_0)=\alpha_{0,0}y_0=\alpha_{-1,-1}y_0$. Finally, when $i\neq 0,-1$ and $j=0$, we have\vspa
\begin{eqnarray*}
&&0=2\varphi(y_i\circ y_0)-y_i\circ \varphi(y_0)-\varphi(y_i)\circ y_0=2\sum_{k = -1}^{p^n - 2} \alpha_{i, k}y_{k}-\alpha_{0, 0}y_{i}-\alpha_{i,-1}by_{p^n-2}-\sum_{k = -1}^{p^n - 2} \alpha_{i, k}y_{k}.
\end{eqnarray*}
Then $\varphi(y_i)=\alpha_{i,i}y_i=\alpha_{0,0}y_i$. Hence, $\varphi(y_i)=\alpha_{0,0}y_i$, for all $i=-1,0,...,p^n-2$, i.e., $\varphi$ is trivial.
\end{proof}

\begin{cor}
Let $(A,\circ )$ be a finite-dimensional simple Novikov algebra over an algebraically closed field  $\mathbf{k}$ with characteristic $p > 2$, where $\text{dim}(A)\geq 2$. Then there are no non-trivial {\transNovikovpoisson algebra} structures on $(A,\circ)$.
\vspa
\end{cor}
\begin{proof}
It follows directly from Theorems~\ref{1/2-trans} and~\ref{finitesimplethm}.
\end{proof}

Let $\mathbf{k}$ be a field of characteristic $0$, and let $(A,\circ)$ be an infinite-dimensional simple Novikov algebra over $\mathbf{k}$ containing a nonzero element $e\in A$, such that the left multiplication operator $L_\circ(e)$ satisfies:
\begin{eqnarray*}
L_\circ(e)(e) \in \mathbf{k}e, \quad \dim(\text{span}\{(L_\circ(e))^n(x) \mid n \in \mathbb{N}\})< \infty \quad \text{for any } x \in A.
\end{eqnarray*}
By Osborn's classification in \cite{simplewuxian} and its simplified formulas, there exists an additive subgroup $\Delta$ of $\mathbf{k}$ and a basis $\{x_{\alpha,j} \mid \alpha \in \Delta, \, j \in J\}$ with $J = \{0\}$ or $\mathbb{N}$, $\{0\} \neq \Delta + J \subset \mathbf{k}$, such that
\begin{eqnarray}
&&\label{infintefomula}x_{\alpha, i}\circ x_{\beta, j} = (\beta + b)x_{\alpha + \beta, i+j} +j x_{\alpha + \beta, i+j - 1}.
\end{eqnarray}

\begin{thm}\label{infinitesimplethm}
Let $(A,\circ)$ be an infinite-dimensional simple Novikov algebra defined above. Then there are no non-trivial $\frac{1}{2}$-derivations of $(A,\circ)$.
\vspa
\end{thm}

\begin{proof}
Let $\varphi$ be a $\frac{1}{2}$-derivation of $(A,\circ)$. For all $\alpha \in \Delta$ and $i\in \mathbb{N}$, we set \vspa
\begin{eqnarray}
&&\label{infinitesimple1/2}\varphi(x_{\alpha,i})=\sum_{\lambda\in \Delta,l\in \mathbb{N}}a_{\alpha,i}^{\lambda,l}x_{\lambda,l},
\end{eqnarray}
where $a_{\alpha,i}^{\lambda,l}\in {\bf k}$.\par

%%¦Ì¨²¨°?¨¤¨¤
Case 1: When \(\Delta = 0\) and $J=\mathbb{N}$,
we simply denote $x_{0,i}$ and $a_{0,i}^{0,j}$ by $x_i$ and $a_{i,j}$ respectively. So Eqs.~\eqref{infintefomula} and~\eqref{infinitesimple1/2} are equivalent to $x_i\circ x_j=bx_{i+j}+jx_{i+j-1}$ and $\varphi(x_i)=\sum_{j\in \mathbb{N}}a_{i,j}x_j$ respectively. Let $i=j=0$. Then we have\vspa
\begin{eqnarray*}
&&0=2\varphi(x_0\circ x_0)-x_0\circ \varphi(x_0)-\varphi(x_0)\circ x_0=2b\sum_{j\in \mathbb{N}}a_{0,j}x_j-2b\sum_{j\in \mathbb{N}}a_{0,j}x_j-\sum_{j\in \mathbb{N}}a_{0,j}jx_{j-1},
\end{eqnarray*}
which implies that $\varphi(x_0)=a_{0,0}x_0$.\par
If $b\neq0$, we consider\vspa
\begin{eqnarray*}
&&0=2\varphi(x_i\circ x_0)-x_i\circ \varphi(x_0)-\varphi(x_i)\circ x_0=2b\sum_{j\in \mathbb{N}}a_{i,j}x_j-ba_{0,0}x_i-b\sum_{j\in \mathbb{N}}a_{i,j}x_{j},
\end{eqnarray*}
which implies that $\varphi(x_i)=a_{i,i}x_i=a_{0,0}x_i$, for all $i\in \mathbb{N}$. Thus, $\varphi$ is trivial.\par
If $b=0$, then Eq.~\eqref{infintefomula} is equivalent to $x_i\circ x_j=jx_{i+j-1}$. Set $i=j=1$. Then we have\vspa
\begin{eqnarray*}
&&0=2\varphi(x_1\circ x_1)-x_1\circ \varphi(x_1)-\varphi(x_1)\circ x_1=2\sum_{j\in \mathbb{N}}a_{1,j}x_j-\sum_{j\in \mathbb{N}}a_{1,j}jx_{j}-\sum_{j\in \mathbb{N}}a_{1,j}x_{j},
\end{eqnarray*}
which implies that $\varphi(x_1)=a_{1,1}x_1$. Next, we consider\vspa
\begin{eqnarray*}
&&0=2\varphi(x_i\circ x_1)-x_i\circ \varphi(x_1)-\varphi(x_i)\circ x_1=2\sum_{j\in \mathbb{N}}a_{i,j}x_j-a_{1,1}x_{i}-\sum_{j\in \mathbb{N}}a_{i,j}x_{j},
\end{eqnarray*}
which implies that $\varphi(x_i)=a_{i,i}x_i=a_{1,1}x_i$, for all $i\in \mathbb{N}$. Hence, $\varphi$ is trivial.\par

%%¦Ì¨²?t¨¤¨¤
Case 2: When $\Delta\neq0$ and $J=0$, we simply denote $x_{\alpha,i}=x_{\alpha}$ and $a_{\alpha,0}^{\lambda,0}=a_{\alpha,\lambda}$. So Eqs.~\eqref{infintefomula} and~\eqref{infinitesimple1/2} are equivalent to $x_\alpha\circ x_\beta=(\beta+b)x_{\alpha+\beta}$. and $\varphi(x_{\alpha})=\sum_{\lambda\in \Delta }a_{\alpha,\lambda}x_{\lambda}$ respectively. Let $\alpha=\beta=0$ in Eq.~\eqref{infintefomula}. Then we have \vspa
\begin{eqnarray*}
&&0=2\varphi(x_0\circ x_0)-x_0\circ \varphi(x_0)-\varphi(x_0)\circ x_0=2b\sum_{\lambda\in \Delta}a_{0,\lambda}x_{\lambda}-\sum_{\lambda\in \Delta}a_{0,\lambda}(\lambda+b)x_{\lambda}-\sum_{\lambda\in \Delta}a_{0,\lambda}bx_{\lambda},
\end{eqnarray*}
which implies that $\varphi(x_0)=a_{0,0}x_0$.\par
If $b\neq0$, we consider\vspa
\begin{eqnarray*}
&&0=2\varphi(x_{\alpha}\circ x_0)-x_{\alpha}\circ \varphi(x_{0})-\varphi(x_{\alpha})\circ x_0=2b\sum_{\lambda\in \Delta}a_{\alpha,\lambda}x_{\lambda}-a_{0,0}bx_{\alpha}-\sum_{\lambda\in \Delta}a_{\alpha,\lambda}bx_{\lambda},
\end{eqnarray*}
which implies that $\varphi(x_{\alpha})=a_{\alpha,\alpha}x_{\alpha}=a_{0,0}x_{\alpha}$, for all $\alpha\in \Delta$. Hence, $\varphi$ is trivial.

If $b=0$, then Eq.~\eqref{infintefomula} is equivalent to $x_\alpha\circ x_\beta=\beta x_{\alpha+\beta}$. For all $\alpha\in \Delta$ with $\alpha\neq0$, we have\vspa
\begin{eqnarray*}
&&0=2\varphi(x_{0}\circ x_{\alpha})-x_{0}\circ \varphi(x_{\alpha})-\varphi(x_{0})\circ x_{\alpha}=2\alpha \sum_{\lambda\in \Delta}a_{\alpha,\lambda}x_{\lambda}-\sum_{\lambda\in \Delta}\lambda a_{\alpha,\lambda} x_{\lambda}-\alpha a_{0,0}x_{\alpha},
\end{eqnarray*}
which implies that $\varphi(x_{\alpha})=a_{\alpha,\alpha}x_{\alpha}+a_{\alpha,2\alpha}x_{2\alpha}=a_{0,0}x_{\alpha}+a_{\alpha,2\alpha}x_{2\alpha}$, for all $\alpha\in \Delta$ with $\alpha\neq0$. Next, we consider\vspa
\begin{eqnarray*}
&&0=2\varphi(x_{\alpha}\circ x_{\beta})-x_{\alpha}\circ \varphi(x_{\beta})-\varphi(x_{\alpha})\circ x_{\beta}\\
&&=2\beta(a_{\alpha+\beta,\alpha+\beta}x_{\alpha+\beta}+a_{\alpha+\beta,2(\alpha+\beta)}x_{2(\alpha+\beta)})-\beta a_{\beta,\beta}x_{\alpha+\beta}-2\beta a_{\beta,2\beta}x_{\alpha+2\beta}-\beta a_{\alpha,\alpha}x_{\alpha+\beta}-\beta a_{\alpha,2\alpha}x_{2\alpha+\beta},
\end{eqnarray*}
for all $\alpha,\beta\in \Delta$ with $\alpha,\beta,\alpha+\beta \neq 0$. Then by comparing the coefficient of $x_{2(\alpha+\beta)}$, we obtain $a_{\lambda,2\lambda}=0$  for all $\lambda \in \Delta$ with $\lambda\neq 0$. Hence, $\varphi$ is trivial.\par

%%¦Ì¨²¨¨y¨¤¨¤
Case 3: When $\Delta\neq0$ and $J=\mathbb{N}$, then $|\Delta|=\infty$ since $\operatorname{char}({\bf k})=0$. First, let $\alpha=\beta=0$ and $i=j=0$ in Eq.~\eqref{infintefomula}. Then
\begin{eqnarray*}
0&=&2\varphi(x_{0,0} \circ x_{0,0})-x_{0,0} \circ \varphi(x_{0,0})-\varphi(x_{0,0}) \circ x_{0,0}\\
&=&2b\sum_{\lambda\in \Delta,l\in \mathbb{N}} a_{0,0}^{\lambda,l}x_{\lambda,l}-\sum_{\lambda\in \Delta,l\in \mathbb{N}} a_{0,0}^{\lambda,l} \big( (\lambda + b)x_{\lambda,l} + l x_{\lambda, l- 1} \big)-\sum_{\lambda\in \Delta,l\in \mathbb{N}} a_{0,0}^{\lambda,l} \cdot b x_{\lambda,l}.
\end{eqnarray*}
Comparing the coefficients of $x_{\lambda,l}$, we have
\begin{eqnarray}\label{insimple1}
&&a_{0,0}^{\lambda,l+1} = \frac{-\lambda}{l+1} a_{0,0}^{\lambda,l}.
\end{eqnarray}
Then we obtain $a_{0,0}^{\lambda,l} = \frac{(-\lambda)^l}{l!} a_{0,0}^{\lambda,0}$. In particular, we obtain $a_{0,0}^{0,l} = 0 \ (l \geq 1)$. Since Eq.~\eqref{infinitesimple1/2} is a finite sum,  we have $a_{0,0}^{\lambda,0}=0$, for all $\lambda \in \Delta$ with $\lambda\neq0$, i.e., $\varphi(x_{0,0})=a_{0,0}^{0,0}x_{0,0}$.\par
If $b\neq0$, we consider \vspa
\begin{eqnarray*}
0&=&2\varphi(x_{\alpha,i} \circ x_{0,0})-x_{\alpha,i}\circ \varphi(x_{0,0})-\varphi(x_{\alpha,i})\circ x_{0,0}\\
&=&2b\sum_{\lambda\in \Delta,l\in \mathbb{N}}a_{\alpha,i}^{\lambda,l}x_{\lambda,l}-a_{0,0}^{0,0}bx_{\alpha,i}-\sum_{\lambda\in \Delta,l\in \mathbb{N}}ba_{\alpha,i}^{\lambda,l}x_{\lambda,l},
\end{eqnarray*}
which implies that $\varphi(x_{\alpha,i})=a_{\alpha,i}^{\alpha,i}x_{\alpha,i}=a_{0,0}^{0,0}x_{\alpha,i}$, for all $\alpha\in \Delta,i\in \mathbb{N}$. Thus, $\varphi$ is trivial.\par

If $b=0$, Eq.~\eqref{infintefomula} is equivalent to $x_{\alpha,i}\circ x_{\beta,j}=\beta x_{\alpha+\beta,i+j}+jx_{\alpha+\beta,i+j-1}$. Let $i=j=1$ and $\alpha=\beta=0$. Then we have\vspa
\begin{eqnarray*}
&&0=2\varphi(x_{0,1}\circ x_{0,1})-\varphi(x_{0,1})\circ x_{0,1}-x_{0,1}\circ \varphi (x_{0,1})\\
&&=2\sum_{\lambda\in \Delta,l\in \mathbb{N}}a_{0,1}^{\lambda,l}x_{\lambda,l}-\sum_{\lambda\in \Delta,l\in \mathbb{N}}a_{0,1}^{\lambda,l}x_{\lambda,l}-\sum_{\lambda\in \Delta,l\in \mathbb{N}}a_{0,1}^{\lambda,l}(\lambda x_{\lambda,l+1}+lx_{\lambda,l}).
\end{eqnarray*}
Comparing the coefficients of $x_{0,l}$, $x_{\lambda,0}$ and $x_{\lambda,l+1}$, for all $\alpha \in \Delta$ and $l\in \mathbb{N}$, we obtain that $a_{0,1}^{0,l}$=0 when $l\neq1$, $a_{0,1}^{\lambda,0}=0$ and $la_{0,1}^{\lambda,l+1}=-\lambda a_{0,1}^{\lambda,l}$ for all $l\in \mathbb{N}$ and $\lambda\in \Delta$. Since Eq.~\eqref{infinitesimple1/2} is a finite sum, we have $a_{0,1}^{\lambda,l}=0$ when $l,\lambda \neq 0$. Thus, we have $\varphi(x_{0,1})=a_{0,1}^{0,1}x_{0,1}$. Next, we consider\vspa
\begin{eqnarray*}
&&0=\varphi(x_{\alpha,i}\circ x_{0,1})-\varphi(x_{\alpha,i})\circ x_{0,1}-x_{\alpha,i}\circ \varphi (x_{0,1})=2\sum_{\lambda\in \Delta,l\in \mathbb{N}}a_{\alpha,i}^{\lambda,l}x_{\lambda,l}-\sum_{\lambda\in \Delta,l\in \mathbb{N}}a_{\alpha,i}^{\lambda,l}x_{\lambda,l}-a_{0,1}^{0,1}x_{\alpha,i},
\end{eqnarray*}
which implies that $\varphi(x_{\alpha, i})=a_{\alpha,i}x_{\alpha,i}=a_{0,1}^{0,1}x_{\alpha,i}$, for all $\alpha \in \Delta$ and $i\in \mathbb{N}$. Hence, $\varphi$ is trivial.
\vspa
\end{proof}

\begin{cor}
Let $(A,\circ)$ be the infinite-dimensional simple Novikov algebra over a field ${\bf k}$ of characteristic $0$ defined by Eq. (\ref{eq-simple-Nov}). Then there are no non-trivial {\transNovikovpoisson algebra} structures on $(A,\circ)$.
\end{cor}
\begin{proof}
\vspa
It follows directly from Theorems~\ref{1/2-trans} and~\ref{infinitesimplethm}.
\end{proof}

\noindent {\bf Acknowledgments.}
This research is supported by
Natural Science Foundation of China (No. 12171129) and  Zhejiang
Provincial Natural Science Foundation of China (No. Z25A010006).

\smallskip

\noindent
{\bf Declaration of interests. } The authors have no conflicts of interest to disclose.

\smallskip

\noindent
{\bf Data availability. } No new data were created or analyzed in this study.

\end{document}